\documentclass{article}
\usepackage{geometry}
\usepackage{comment}
\usepackage{mathrsfs}
\usepackage{mathscinet}
\usepackage{hyperref}

\usepackage{amsmath}
\usepackage{amsfonts}
\usepackage{amssymb}
\usepackage{amsthm}
\usepackage{wasysym}
\usepackage{graphicx, enumerate, caption, subcaption}
\usepackage{color}
\providecommand{\U}[1]{\protect\rule{.1in}{.1in}}
\newtheorem{thm}{Theorem}
\newtheorem{lem}{Lemma}

\newtheorem{prop}{Proposition}

\newtheorem{lemma}[lem]{Lemma}

\newtheorem{proposition}[prop]{Proposition}

\theoremstyle{remark}

\renewenvironment{proof}[1][Proof]{\noindent\textbf{#1.} }{\ \rule{0.5em}{0.5em}}
\begin{document}


\title{Rates of Convergence to Stationarity for Reflected Brownian Motion}

\author{%
Jose Blanchet\\
Stanford University\\
\texttt{jose.blanchet@stanford.edu}
\and
Xinyun Chen\\
The Chinese University of Hong Kong, Shenzhen\\
\texttt{chenxinyun@cuhk.edu.cn}%
}

\date{}

\maketitle

\begin{abstract}
We provide the first rate of convergence to stationarity analysis for reflected Brownian motion (RBM) as the dimension
 grows under some uniformity conditions. In particular, if the underlying
 routing matrix is uniformly contractive, uniform stability of the drift
 vector holds, and the variances of the underlying Brownian Motion (BM) are
 bounded, then we show that the RBM converges exponentially fast to
 stationarity with a relaxation time of order $O\left( d^{4}(\log(d))^{3}\right)$ as the dimension $d\rightarrow\infty$. Our bound for the relaxation time follows as a corollary of the non-asymptotic bound we obtain for the initial transient effect, which is explicit in terms of the RBM parameters.

\end{abstract}

\section{Introduction}

Multidimensional Reflected Brownian Motion (RBM) was introduced by Harrison
and Reiman in \cite{HarrisonReiman_1981}. It is one of the most important
models in Operations Research because it can be used to approximate in
distribution a very large class of stochastic networks of interest as the
traffic utilization of the system approaches 100\% (i.e. in heavy traffic).
For example, \cite{HarrisonWilliams_1987b} explains in Section 2 that the
Harrison-Reiman RBM can be used to approximate the workload of virtually any
open network of homogeneous customers under a work-conserving policy under
mild weak convergence conditions on the incoming traffic. See also Chapter 7
of \cite{ChenYao2001} and the references therein. Moreover, it has been shown
that the approximation holds also for the underlying steady-state
distributions in significant generality (see \cite{BudhirajaLee_2009} and
\cite{GamarnikZeevi_2006}). The Harrison-Reiman RBM can also be used, in some special cases, to approximate feedforward networks of heterogeneous customers, see the discussion on page 270, before Theorem 3.1 of \cite{Harrison_1992}. 

In this paper, we study the rate of convergence to stationarity of the
multidimensional Harrison-Reiman RBM (or simply RBM throughout the rest of our
development). We provide the first rate of convergence analysis for RBM as the
dimension $d$ grows under certain uniformity conditions. In particular, if the
underlying routing matrix is uniformly contractive (see Assumption A1),
uniform stability of the drift vector holds (see Assumption A2), and the
variances of the underlying Brownian Motion (BM) are bounded (see Assumption
A3), then we show that the RBM converges exponentially fast to stationarity
with a relaxation time of order $O\left(  d^{4}(\log\left(  d\right)
)^{3}\right)  $ as $d\rightarrow\infty$.

Our result is related to the literature on geometric ergodicity of RBMs. There are two papers \cite{BudhirajaLee_2007} and \cite{Sarantsev_2017} that are mostly close to our work. They established geometric ergodicity results for the same type of RBMs studied in our paper. (Actually, their results cover a more general class of stochastic processes than ours.) But they didn't provide an explicit expression of convergence rate in terms of the RBM parameters or discuss the dependence of the convergence rate on the number of dimensions. In Section 3.2 of \cite{Sarantsev_2017}, an explicit expression of the Lyapunov function is given for RBMs on the positive orthant but with extra restrictions on the reflection matrix of the RBM.

We believe that the results and the techniques introduced in this paper can be
used in the design of polynomial-time (as $d$ increases) Monte Carlo methods
for approximating the steady-state distribution of RBM. This is because, in such algorithms, one needs to know how long to run the simulation to guarantee that the process has reached or become very ``close" to stationarity. Of course, one can heuristically monitor the convergence, but unless one has quantitative rates of convergence, it is difficult to be completely sure that initial transient effects have been controlled. The results in this paper provide explicit guidelines about how long to run simulations for high dimensional stochastic networks modeled via RBM. The paper \cite{MLMC} actually goes beyond this direct implication and uses the analysis in this paper to design efficient Monte Carlo methods for steady-state estimation of RBM. These algorithms, as far as we understand, would be the first of their kind in the
context of numerical methods for the steady-state distribution of RBM,
although polynomial-time convergence algorithms for steady-state simulation of
reflected processes (of compound Poisson type) have been studied in
\cite{BlanchetChen_2015}.

In Section \ref{Section_Notation}, we first introduce our notation and provide
the statement of our main result. Also in Section \ref{Section_Notation}, we
provide a step-by-step outline of the proof of our result. The proof is
divided into three steps, which are developed throughout Sections
\ref{Section_Step_1} to \ref{Section_Step_3}. Finally, in Section
\ref{Sec: conclude}, we discuss the main bottlenecks that need to be overcome
in our arguments in order to improve upon the bounds that we obtain.

\section{Notation, Assumptions and Main Result\label{Section_Notation}}

We start this section by introducing the notation, and then we explain the motivation and definition of RBM and
the assumptions that we shall impose throughout the paper. We concentrate on
the case where $d\geq2$, and the case in which $d=1$ is standard.

\subsection{Notation}

For convenixence, we summarize the common notations used through out the paper.
We shall use boldface to write vector quantities, which are encoded as
columns. For instance, we write $\mathbf{y}=\left(  y_{1},...,y_{d}\right)
^{T}$. We use $\mathbf{1}$ to denote the vector with all entries equal to
unity. We define the following norms of vectors: $\left\Vert \mathbf{y}%
\right\Vert _{\infty}=\max_{i=1}^{d}|y_{i}|$ and $\left\Vert \mathbf{y}%
\right\Vert _{1}=\sum_{i=1}^{d}|y_{i}|$. Let $\partial \mathbb{R}^d_+$ be the boundary of the positive orthant.

We write $1(\cdot)$ to denote the indicator function and $I$ to the identity
matrix. For a $d\times d$ matrix $A$, we let $A^{T}$ be its transposition. For
any subsets $S_{1}$ and $S_{2}$ of $\{1,2,...,d\}$, we write $A_{S_{1}S_{2}}$
as the submatrix of $A$ such that $A_{S_{1}S_{2}}=\{A_{ij}:i\in S_{1},j\in
S_{2}\}$. Similarly, $\mathbf{y}_{S_{1}}=(y_{i}:i\in S_{1})$ and $A_{S_{1}%
}=\{A_{ij}:i\in S_{1},1\leq j\leq d\}$.

All inequalities involving vectors or matrices are understood componentwise.
For example, $\mathbf{y}\geq\mathbf{z}$ means that $y_{i}\geq z_{i}$ for all
$i\in\{1,2,...,d\}$.

For any subset $S$ of $\{1,2,...,d\}$, $\bar{S}$ represents its compliment
set, i.e., $\bar{S}=\{1\leq i\leq d:i\notin S\}$. For all $1\leq i,j\leq d$,
$\delta_{ij}$ is the Kronecker delta, i.e., $\delta_{ij}=1$ if $i=j$, and
$\delta_{ij}=0$ if $i\neq j$. The arrow \textquotedblleft$\Longrightarrow
$\textquotedblright\ represents convergence in distribution. The equality
$A\overset{D}{=}B$ means that $A$ and $B$ are equal in distribution. We use
$N(0,1)$ to refer to a generic standard normal random variable.

\subsection{Motivation, Definition of RBM, and Assumptions}

Let us consider the stochastic fluid network model introduced by
\cite{Kella_1996}. It is a network of $d$ queueing stations indexed by
$\{1,2,...,d\}$. Jobs arrive to the network according to some counting process
$(N\left(  t\right)  :t\geq0)$. The $k$-th arrival brings a vector of job
requirements $\mathbf{W}\left(  k\right)  =\left(  W_{1}\left(  k\right)
,...,W_{d}\left(  k\right)  \right)  ^{T}$, which adds $W_{i}(k)$ units of
workload to the $i$-th station right at the moment of arrival, for
$i\in\{1,...,d\}$.

From the previous description, we know that the total amount of work that
arrives to the $i$-th station, up to and including time $t$, is denoted by
\begin{equation}
J_{i}\left(  t\right)  =\sum_{k=1}^{N\left(  t\right)  }W_{i}\left(  k\right)
. \label{eq: J}%
\end{equation}
Let us now assume that for all $i\in\{1,...,d\}$, the server of station $i$
processes the workload as a fluid at rate $r_{i}>0$. That means, if the
workload in the $i$-th station remains strictly positive during the time
interval $[t,t+h]$, the output from station $i$ during this time interval will
be $r_{i}h$. In addition, for all $1\leq i,j\leq d$, let $Q_{i,j}\geq0$ be the
proportion of the fluid circulated to the $j$-th station, after being
processed by the $i$-th server. The matrix $Q=\left(  Q_{i,j}:1\leq i,j\leq
d\right)  $ is called the routing matrix of the network. Without loss of
generality, we assume that $Q_{i,i}=0$. We introduce an extra notation
$Q_{i,0}=1-\sum_{j=1}^{d}Q_{i,j}\geq0$ to represent the proportion of the
fluid that leaves the network immediately after being processed by the $i$-th
sever. Note that the matrix $Q$ does not include $Q_{i,0}$.

It is natural to assume that arriving jobs will eventually leave the network,
which is equivalent to assuming that $Q^{n}\rightarrow0$ as $n\rightarrow
\infty$; which, in turn, is equivalent to requiring that $Q$ be a strict
contraction in the sense that it has a spectral radius which is strictly less
than one. In other words, one assumes there exists $\beta\in\left(
0,1\right)  $ and $\kappa\in\left(  0,\infty\right)  $ such that:
\begin{equation}
\left\Vert \mathbf{1}^{T}Q^{n}\right\Vert _{\infty}\leq\kappa\left(
1-\beta\right)  ^{n}. \label{Sp_Radius}%
\end{equation}

The dynamics of such a stochastic fluid network can be expressed formally in
differential notation as follows. Let $Y_{i}\left(  t\right)  $ denote the
workload content of the $i$-th station at time $t$, then given $Y_{i}\left(
0\right)  $, we write:%
\begin{align}
dY_{i}\left(  t\right)   &  =dJ_{i}\left(  t\right)  -r_{i}I\left(
Y_{i}\left(  t\right)  >0\right)  dt+\sum_{j:j\neq i}Q_{j,i}r_{j}I\left(
Y_{j}\left(  t\right)  >0\right)  dt\label{S1_b}\\
&  =dJ_{i}\left(  t\right)  -r_{i}dt+\sum_{j:j\neq i}Q_{j,i}r_{j}%
dt+r_{i}I\left(  Y_{i}\left(  t\right)  =0\right)  dt\nonumber\\
&  ~~-\sum_{j:j\neq i}Q_{j,i}r_{j}I\left(  Y_{j}\left(  t\right)  =0\right)
dt\nonumber
\end{align}
for $i\in\{1,...,d\}$. These equations take a neat form in matrix notation.
Let $\mathbf{r}=\left(  r_{1},...,r_{d}\right)  ^{T}$ be the column vector
corresponding to the service rates, and define the so-called \textquotedblleft
reflection matrix\textquotedblright\ as $R=\left(  I-Q\right)  ^{T}$. Let%
\begin{equation}
\mathbf{X}\left(  t\right)  =\mathbf{J}\left(  t\right)  -R\mathbf{r}t,
\label{Aux_Input_1}%
\end{equation}
where $\mathbf{J}\left(  t\right)  $ is a column vector with its $i$-th
component equal to $J_{i}\left(  t\right)  $ as defined in \eqref{eq: J}.
Then, we can see from (\ref{S1_b}) that in the special case of input path
$\mathbf{X}\left(  \cdot\right)  $, given by (\ref{Aux_Input_1}),
$\mathbf{Y}\left(  \cdot\right)  $ solves the so-called Skorokhod problem,
which is posed as follows:

\textbf{Skorokhod Problem: }\textit{Given a process }$\mathbf{X}\left(
\cdot\right)  $\textit{\ and a matrix }$R$\textit{, we say that the pair
}$(\mathbf{Y},\mathbf{L})$\textit{\ solves the associated Skorokhod problem
if}%
\begin{equation}
0\leq\mathbf{Y}\left(  t\right)  =\mathbf{Y}\left(  0\right)  +\mathbf{X}%
\left(  t\right)  +R\mathbf{L}\left(  t\right)  , ~\mathbf{L}(0)=0 \label{SP1}%
\end{equation}
\textit{where the }$i$\textit{-th entry of }$\mathbf{L}\left(  \cdot\right)
$\textit{\ is non-decreasing and }$\int_{0}^{t}Y_{i}\left(  s\right)
dL_{i}\left(  s\right)  =0$\textit{.}

The seminal paper \cite{HarrisonReiman_1981} shows that the Skorokhod problem
is well posed (i.e. it has a unique solution) in the case where the input
$\mathbf{X}\left(  \cdot\right)  $ is continuous and $R$ is a so-called
$M$-matrix. In particular, a matrix $R$ is said to be an $M$-matrix if%
\begin{equation}
R^{-1}\text{ exists and it has non-negative entries.} \label{M_Cond}%
\end{equation}
In our case, $\mathbf{X}(\cdot)$ is a multi-dimensional Brownian motion with
drift vector $\boldsymbol{\mu}$ and covariance matrix $\Sigma:=CC^{T}$, and
hence it is continuous almost surely. The reflection matrix $R=(I-Q)^{T}$ is
indeed an $M$-matrix. The unique solution to the Skorokhod problem when the
input is a $\left(  \boldsymbol{\mu},\Sigma\right)  $-Brownian Motion is
called a $\left(  \boldsymbol{\mu},\Sigma,R\right)  $-RBM.

To understand intuitively why the $M$-condition assumption is very natural,
once again we go back to the stochastic fluid network depicted in (\ref{S1_b})
and note that $R=I-Q^{T}$ being an $M$-matrix is equivalent to requiring that
(\ref{Sp_Radius}) holds.

To appreciate the delicate nature of $\mathbf{L}\left(  \cdot\right)  $, note
that in the setting of the stochastic fluid network depicted in (\ref{S1_b})
we have that%

\begin{equation}
L_{i}\left(  t\right)  =\int_{0}^{t}r_{i}I\left(  Y_{i}\left(  s\right)
=0\right)  ds. \label{L_constr}%
\end{equation}

For general Skorokhod problems, under the $M$-condition and some mild
conditions on $\mathbf{X}\left(  \cdot\right)  $, the assumption that
\begin{equation}
R^{-1}E\mathbf{X}\left(  1\right)  <0, \label{Stability_Cond}%
\end{equation}
implies that $\mathbf{Y}\left(  t\right)  \Longrightarrow\mathbf{Y}\left(
\infty\right)  $ as $t\rightarrow\infty$, where $\mathbf{Y}\left(
\infty\right)  $ is a random variable with the (unique) stationary
distribution of $\mathbf{Y}\left(  \cdot\right)  $. In particular, according
to \cite{HarrisonWilliams_1987b}, condition (\ref{Stability_Cond}) is
necessary and sufficient for stability of the $\left(  \boldsymbol{\mu}%
,\Sigma,R\right)  $-RBM (i.e. a unique stationary distribution exists) under
the $M$-condition (\ref{M_Cond}).

In this paper, we shall consider a family of $\left(  \boldsymbol{\mu}%
,\Sigma,R\right)  $-RBMs indexed by the dimension $d$. Implicitly, then, $R$,
$\boldsymbol{\mu}$, and $\Sigma$ are indexed by their dimension. Our goal is
to derive rates of convergence to stationarity that behave graciously as
$d\rightarrow\infty$ under suitable uniformity conditions, which are stated in
the following assumptions.

\bigskip

\textbf{Assumptions:}

\textbf{A1) Uniform contraction:} We let $R=I-Q^{T}$, where $Q$ is
substochastic and assume that there exists $\beta_{0}\in\left(  0,1\right)  $
and $\kappa_{0}\in\left(  0,\infty\right)  $ independent of $d$ such that%
\begin{equation}
\left\Vert \mathbf{1}^{T}Q^{n}\right\Vert _{\infty}\leq\kappa_{0}(1-\beta
_{0})^{n}. \label{Ass_A1}%
\end{equation}
Under (\ref{Ass_A1}) we observe that
\[
\left\Vert R^{-1}\mathbf{1}\right\Vert _{\infty}\leq b_{1}:=\kappa_{0}%
/\beta_{0}<\infty.
\]

\textbf{A2) Uniform stability:} We write $\mathbf{X}\left(  t\right)
=\boldsymbol{\mu}t+C\mathbf{B}\left(  t\right)  $, where $\mathbf{B}\left(
t\right)  =$ $(B_{1}\left(  t\right)  ,...,B_{d}\left(  t\right)  )^{T}$ and
the $B_{i}\left(  \cdot\right)  $'s are standard Brownian motions, and the
matrix $C$ satisfies $\Sigma=CC^{T}$. We assume that there exists $\delta
_{0}>0$ independent of $d$ such that%
\[
R^{-1}\boldsymbol{\mu}<-\delta_{0}\mathbf{1}.
\]

\textbf{A3) Uniform marginal variability:} Define $\sigma_{i}^{2}=\Sigma
_{i,i}$ (i.e. the variance of the $i$-th coordinate of $\mathbf{X}$). We
assume that there exists $b_{0}\in\left(  0,\infty\right)  $, independent of
$d\geq1$, such that
\[
b_{0}^{-1}\leq\sigma_{i}^{2}\leq b_{0}.
\]

\bigskip

An important constant to be used in the sequel is $\delta
_{1}=\delta_{0}\beta_{0}/(2\kappa_{0})$. This constant will be used in the
introduction of a useful dominating process.

\bigskip

Let us now discuss the nature of our assumptions in terms of the structure of the queueing network that the RBM approximates in heavy traffic. Assumption A1) means that the expected number of stations that a customer, who enters from any station $i$, will visit before leaving the system is bounded in a suitable sense as the size of the network grows. In particular, we can interpret $Q$ as the transition matrix of an absorbing Markov chain (as we define and denote by $M$ in Section \ref{Section_Step_1}). The absorbing state can be interpreted as leaving the network. If we place a customer uniformly at random in any station (denote this distribution by $\upsilon _{d}$) and let $\rho $ be the number of
transitions of the chain (i.e. the number of stations) before being absorbed (i.e. before leaving the network). Then, Assumption A1) could be replaced by the more easy-to-interpret assumption that $P_{\upsilon_{d}}\left( \rho >n\right) \leq \kappa _{0}\left( 1-\beta _{0}\right) ^{n}$.  For example, all networks in which customers will leave the system  with positive probability $\beta_0$ after been served at each station will satisfy Assumption A1). In this case, the number of stations that the customer will visit before leaving will be stochastically bounded by a geometric random varaible with mean $1/ \beta_0$.  Other examples could be large networks in which the number of stations a customer will visit is much smaller in expectation or very light-tailed than the size of the network.  A family of networks which seem to exhibit these features are those arising in the queueing models for patient flow in hospitals (see for example \cite{Armony_2018} and \cite{Creemer_2010}). Assumption A2) means that all stations in all the networks reach heavy traffic in a uniform way, i.e. $r_i - \sum_{j}R^{-1}_{ji}E[J(1)]> \delta_0/\sqrt{n}$ for all station $i$. The lower bound in A3) simply avoids
degeneracies. Assumptions A2) and A3) guarantee  the tightness of the
marginal steady-state distributions of the workload in each station in a uniform way. 

\subsection{The Main Result: Statement}

In order to quantify the rate of convergence to stationarity of RBM, we shall
use Wasserstein's distance. Let us define
\[
\mathcal{L}=\{f:\mathbb{R}^{d}\rightarrow \mathbb{R}\text{ such that }\left\vert f\left(
x\right)  -f\left(  y\right)  \right\vert \leq\left\Vert x-y\right\Vert
_{\infty}\}.
\]
In other words, $\mathcal{L}$ is the set of Lipschitz continuous functions on
$R^{d}$ with the Lipschitz constant equal to one under the uniform norm.
Suppose that the random variable $\mathbf{U}\in R^{d}$ has distribution
$\upsilon$ in $R^{d}$ and that $\mathbf{V}\in R^{d}$ has distribution $\varpi
$. The associated Wasserstein distance (of order 1) between $\upsilon$ and
$\varpi$ is defined as
\[
d_{W}\left(  \upsilon,\varpi\right)  =\sup_{f\in\mathcal{L}}\left\vert
Ef\left(  \mathbf{U}\right)  -Ef\left(  \mathbf{V}\right)  \right\vert .
\]
With a slight abuse of notation, we shall actually write $d_{W}\left(
\mathbf{U},\mathbf{V}\right)  $ instead of $d_{W}\left(  \upsilon
,\varpi\right)  $. We have chosen the Wasserstein distance of order 1 because
in the stochastic network setting (which provides some of the main
applications motivating the use of RBM), Lipschitz continuous functions of the
underlying process are natural quantities to study. Examples of these
functions include the maximum workload and the total workload in a subset of
stations in the network. Our results, therefore, allow us to immediately
quantify initial transient errors in expectations of this sort.

Our main result is the following:

\begin{thm}
\label{Thm_Main}Suppose $\mathbf{Y}(0)=\mathbf{y}\in \mathbb{R}_+^d$. Under assumptions A1) to A3),
for any $\beta\in\left(  0,\min(\beta_{0},1/3\right)  \cdot1/3)$ satisfying,%
\begin{equation}
P\left(  N\left(  0,1\right)  <\sqrt{b_{0}}(\delta_{0}-b_{1}^{2})\right)
\geq\beta/d, \label{Geo_1}%
\end{equation}
we have that for all $d$ large enough and each $t>0$,%
\begin{align}
d_{W}\left(  \mathbf{Y}\left(  t\right)  ,\mathbf{Y}\left(  \infty\right)
\right)   &  \leq3\cdot d^{2}\cdot\exp\left(  -\zeta_{1}\cdot\frac{t}{(d^{4}%
\log\left(  d\right)  )}\right) \nonumber\label{eq: thm1}\\
&  \cdot\left(  \kappa_{0}\cdot\left\Vert \mathbf{y}\right\Vert _{1}\cdot
\exp\left(  \zeta_{0}\cdot\frac{\left\Vert \mathbf{y}\right\Vert _{\infty}%
}{d^{3}\log\left(  d\right)  }\right)  +\frac{\kappa_{0}}{\delta_{0}\beta_{0}}b_{0}b_1\right)  ,
\end{align}
where $\zeta_{0}$ and $\zeta_{1}$ are two constants independent of $d$:
\[
\zeta_{0}=\frac{\delta_{1}\cdot\beta^{2}}{2\max_{i=1}^{d}\sigma_{i}^{2}%
},\text{ \ }\zeta_{1}=\frac{\delta_{1}^{2}\cdot\beta^{2}}{16\max_{i=1}%
^{d}\sigma_{i}^{2}}.
\]
In particular, the relaxation time of RBM is of order $O\left(  d^{4}\left(
\log\left(  d\right)  \right)  ^{3}\right)  $ if $\|\mathbf{y}\|_1 = O(1)$. (The relaxation time
is the minimum of $t$ such that
$d_{W}\left(  \mathbf{Y}\left(  t \right)  ,\mathbf{Y}\left(
\infty\right)  \right)  \leq \frac{1}{2} d_W(\mathbf{Y}(0), \mathbf{Y}(\infty))$.)
\end{thm}

\bigskip

\textbf{Remark:} As the bound \eqref{eq: thm1} is explicit, we can actually relax the uniform assumptions A1) to A3) allowing the constants $\beta_0$, $\kappa_0$ and $\delta_0$ to increase with $d$, and obtain a similar bound in which these constants are functions of $d$. However, we can not set the constants to be any functions of $d$. This is because, for Theorem 1 to hold, $b_0$ and $b_1$ must be chosen such that  (\ref{Geo_1}) is satisfied, which is crucial to the construction of $\chi(\theta)$ in the Lyapunov function in Lemma \ref{Lem_LB_CT} and to the boundedness of the number of geometric trials in Step 2.2 in our proof, as we shall explain later in Section \ref{Sect_Outline_Strategy}. As $P\left(  N\left(  0,1\right)  <\sqrt{b_{0}}(\delta_{0}-b_{1}^{2})\right)=O(\exp(-b_0 (b^2_1-\delta_0 )^2))$, we must have $b_0(b_1^2-\delta_0)^2= O(\log(d))$. For example, if we relax only Assumption A1),  we can make $b_{1}=O\left(  \log\left(  d\right)
^{1/4}\right)$ and allow $\beta = O(d^{-\gamma})$ for some $\gamma>0$. Then, we still obtain that the relaxation time $= O(d^{4+2\gamma}\log(d)^3)$ is
polynomial in $d$ (assuming that the rest of assumptions remain in place). Besides, we didn't specify the dependence of the initial state $\mathbf{y}$ on $d$. As the bound in \eqref{eq: thm1} is explicit in $\mathbf{y}$, readers can obtain a new bound according to their choice of the growth rate of $\mathbf{y}$ in $d$. For example,  the polynomial relaxation time will still hold if $\|\mathbf{y}\|_\infty
= O(d^3)$ (and hence $\|\mathbf{y}\|_1 = O(d^4)$).

\subsection{The Main Result: Strategy of the
Proof\label{Sect_Outline_Strategy}}

We first explain the main steps in the proof of Theorem \ref{Thm_Main}. All
the details, including the technical lemmas will be given in the following sections.

\textbf{Step 0:} We start by considering a natural coupling. Given the
underlying $\left(  \boldsymbol{\mu},\Sigma\right)  $-Brownian motion
$\mathbf{X}\left(  \cdot\right)  $, we consider the $\left(  \boldsymbol{\mu
},\Sigma,R\right)  $-RBM, $\mathbf{Y}\left(  \cdot\right)  $, obtained by
solving the Skorokhod problem with reflection matrix $R$ in (\ref{SP1}). In
order to emphasize the dependence on the initial condition, we will also write
$\mathbf{Y}\left(  t;\mathbf{Y}\left(  0\right)  \right)  :=\mathbf{Y}\left(
t\right)  $. Now let us use $\mathbf{Y}\left(  \infty\right)  $ to denote a
random variable with the stationary distribution of $\mathbf{Y}\left(
\cdot\right)  $ but independent of $\mathbf{X}\left(  \cdot\right)  $. We then
have, by stationarity, that
\[
\mathbf{Y}\left(  \infty\right)  \overset{D}{=}\mathbf{Y}\left(
t;\mathbf{Y}\left(  \infty\right)  \right)  .
\]
We consider the process $\mathbf{Y}\left(  \cdot;\mathbf{Y}\left(  0\right)
\right)  $ coupled with $\mathbf{Y}\left(  \cdot;\mathbf{Y}\left(
\infty\right)  \right)  $, where the driving signal, $\mathbf{X}\left(
\cdot\right)  $, is common to both processes, but the initial conditions are different.

Note that for any $f\in\mathcal{L}$,
\[
|Ef(\mathbf{Y}\left(  t;\mathbf{Y}\left(  0\right)  \right)  )-Ef(\mathbf{Y}%
\left(  t;\mathbf{Y}\left(  \infty\right)  \right)  )|\leq E\left\vert
\left\vert \mathbf{Y}\left(  t;\mathbf{Y}\left(  0\right)  \right)
-\mathbf{Y}\left(  t;\mathbf{Y}\left(  \infty\right)  \right)  \right\vert
\right\vert _{1}%
\]
and hence
\begin{equation}
d_{W}(\mathbf{Y}\left(  t;\mathbf{Y}\left(  0\right)  \right)  ,\mathbf{Y}%
\left(  t;\mathbf{Y}\left(  \infty\right)  \right)  )\leq E\left\vert
\left\vert \mathbf{Y}\left(  t;\mathbf{Y}\left(  0\right)  \right)
-\mathbf{Y}\left(  t;\mathbf{Y}\left(  \infty\right)  \right)  \right\vert
\right\vert _{1}. \label{Was_1}%
\end{equation}
Therefore, to prove Theorem 1, it suffices to show that
\[
E\left\vert \left\vert \mathbf{Y}\left(  t;\mathbf{Y}\left(  0\right)
\right)  -\mathbf{Y}\left(  t;\mathbf{Y}\left(  \infty\right)  \right)
\right\vert \right\vert _{1}%
\]
can be bounded by the right hand side of (\ref{eq: thm1}). We shall do this
through the following steps.

\textbf{Step 1: }The first step in the proof involves bounding
\[
\left\Vert \mathbf{Y}\left(  t;\mathbf{Y}\left(  0\right)  \right)
-\mathbf{Y}\left(  t;\mathbf{Y}\left(  \infty\right)  \right)  \right\Vert
_{1}.
\]
Next we consider a sequence $\{\eta^{k}(\mathbf{y})\}$ which intuitively are
the stopping times by which the RBM, starting from time $0$ at $\mathbf{y}%
\in\mathbb{R}^{d}$, has experienced $k$ ``rounds" and in each round it hits 0 in all the $d$ dimensions for at least once. In detail, we define $\eta^{0}\left(  \mathbf{y}\right)  =0$ and recursively%
\begin{align}
\eta_{i}^{k}\left(  \mathbf{y}\right)   &  =\inf\{t>\eta^{k-1}\left(
\mathbf{y}\right)  +1:Y_{i}\left(  t;\mathbf{y}\right)  =0\},\label{eq: eta}\\
\eta^{k}\left(  \mathbf{y}\right)   &  =\sup\{\eta_{i}^{k}\left(
\mathbf{y}\right)  :1\leq i\leq d\}.\nonumber
\end{align}
Intuitively, $\eta^k(\mathbf{y})$ is the time by which the RBM has just experienced $k$ rounds. We then define
\[
\mathcal{N}\left(  t;\mathbf{y}\right)  =\sup\{k\geq0:\eta^{k}\left(
\mathbf{y}\right)  \leq t\}
\]
as the number of rounds the RBM has experienced by time $t$. We will show in Lemma \ref{Lem_DB1} that after each round, at least a fixed proportion of initial jobs ($=O( \beta_0) $), assuming that they have lower service priority than all jobs arrive after time 0, have left the system. For example, consider a 2-station network with routing probability matrix 
$$
Q =\left(\begin{matrix}
0&1\\0.9&0\\
\end{matrix}\right).
$$
In a single round, if station 1 goes to 0 first and then station 2 goes to 0, we have that $10\%$ of the initial jobs in station 1 and station 2, respectively, have left. As a consequence, we obtain%
\begin{align}
&  \left\Vert \mathbf{Y}\left(  t;\mathbf{Y}\left(  \infty\right)  \right)
-\mathbf{Y}\left(  t;\mathbf{Y}\left(  0\right)  \right)  \right\Vert
_{1}\label{Der_1}\\
&  \leq\left\Vert \mathbf{Y}\left(  t;\mathbf{Y}\left(  \infty\right)
\right)  -\mathbf{Y}\left(  t;\mathbf{0}\right)  \right\Vert _{1}+\left\Vert
\mathbf{Y}\left(  t;\mathbf{Y}\left(  0\right)  \right)  -\mathbf{Y}\left(
t;\mathbf{0}\right)  \right\Vert _{1}\nonumber\\
&  \leq d\cdot\kappa_{0}\cdot(\left(  1-\beta_{0}\right)  ^{\mathcal{N}\left(
t;\mathbf{Y}\left(  \infty\right)  \right)  }\left\Vert \mathbf{Y}\left(
\infty\right)  \right\Vert _{1}+\left(  1-\beta_{0}\right)  ^{\mathcal{N}%
\left(  t;\mathbf{Y}\left(  0\right)  \right)  }\left\Vert \mathbf{Y}\left(
0\right)  \right\Vert _{1}).\nonumber
\end{align}
%
We obtain \eqref{Der_1} based on some elementary estimates following the
analysis in \cite{KellaRamasubramanian_2012}. We actually apply part (iv)
of Theorem 1 in \cite{KellaRamasubramanian_2012}, which states that if two RBMs are only different in their
intial values, their difference is non-increasing in time. Intuitively, we show in Lemma \ref{Lem:DM} that when
one of the coordinates has hit zero for at least once, the difference
$\mathbf{Y}\left(  t;\mathbf{Y}\left(  \infty\right)  \right)  -\mathbf{Y}%
\left(  t;\mathbf{Y}\left(  0\right)  \right)  $ shrinks by a factor which can
be expressed in terms of a suitable substochastic matrices.

\textbf{Step 2: }Combining (\ref{Was_1})\ and (\ref{Der_1}), it is easy to see
that the key to our estimates involves bounding $E\left[  \left(  1-\beta
_{0}\right)  ^{\mathcal{N}\left(  t;\mathbf{y}\right)  }\right]  $ and
$E[\left\Vert \mathbf{Y}\left(  \infty\right)  \right\Vert _{1}]$.

At this point, we invoke a sample-path upper bound $\mathbf{Y}^{+}%
(t;\mathbf{y})$ (introduced in \cite{Kella_1996}) for $\mathbf{Y}\left(
t;\mathbf{y}\right)  $ and its formal definition is given in Section
\ref{Subsect_tech_pfs}. In particular, $\mathbf{Y}^{+}\left(  \cdot
;\mathbf{y}\right)  $ is also an RBM having the same covariance matrix as
$\mathbf{Y}(t,\mathbf{y})$ and an identity reflection matrix. According to
Lemma 3.1 in \cite{Kella_1996}, $\mathbf{Y}^{+}\left(  \cdot;\mathbf{y}%
\right)  $ dominates $\mathbf{Y}(t;\mathbf{y})$ in the sense that
$R^{-1}\mathbf{Y}^{+}(t;\mathbf{y})\geq R^{-1}\mathbf{Y}\left(  t;\mathbf{y}%
\right)  $ for all $t$. Besides, $\mathbf{Y}^+(\cdot;\mathbf{y})$ has a unique
stationary distribution regardless of the initial condition $\mathbf{y}$. Let
$\mathbf{Y}^{+}\left(  \infty\right)  $ follow the stationary distribution of
$\mathbf{Y}^{+}\left(  \cdot\right)  $, then it is well-understood that
$Y_{i}^{+}\left(  \infty\right)  $
follows an exponential distribution with mean $E[Y_{i}^{+}\left(
\infty\right)  ]=\sigma_{i}^{2}/2\left(  \mu_{i}^{+}-\mu_{i}\right)  $
marginally. Therefore, using Assumptions A1) - A3), one can show that
$\sup_{i\geq1}E[Y_{i}^{+}\left(  \infty\right)  ]<\infty$. This upper bound
process, together with Steps 1 and 2, already hints at the polynomial-time
nature of the relaxation time. For example, if $\Sigma$ is diagonal, a
straightforward calculation shows that $E[\max_{1\leq i\leq d}Y_{i}^{+}\left(
\infty\right)  ]=O\left(  \log\left(  d\right)  \right)  $. On the other hand,
starting from equilibrium, in a time interval of order $O\left(  d\right)  $
the maximum coordinate fluctuates at most $O\left(  \log\left(  d\right)
\right)  $ units, while, with very high probability, all coordinates will hit
zero at least once during this time (due to the negative drift of the
underlying Brownian motion driving $\mathbf{Y}^{+}$). One might expect that
the coordinates of the lower bound process would also have visited zero during
this time. However, such a reasoning is not implied by the type of domination
that can be guaranteed between $\mathbf{Y}^{+}\left(  t;\mathbf{y}\right)  $
and $\mathbf{Y}\left(  t;\mathbf{y}\right)  $. In addition, the matrix
$\Sigma$ is not diagonal. So, due to all of these complications, the
quantitative bounds become somewhat involved. The strategy to bound
$E[(1-\beta_{0})^{\mathcal{N}\left(  t;\mathbf{y}\right)  }] $ is split into
several substeps.

\textbf{Step 2.1 (estimating the time to visit a compact): }First, we define
$\tau^{+}\left(  \mathbf{y}\right)  =\inf\{t\geq0:\mathbf{Y}^{+}\left(
t;\mathbf{y}\right)  \leq\mathbf{1}\}$. We define a suitable function
$h(\mathbf{y};\theta)\geq0$ which behaves like $\theta\left\Vert
\mathbf{y}\right\Vert _{\infty}$ for small $\theta$. For each $\theta$ small
enough, we can find $\chi\left(  \theta\right)  >0$ such that%
\[
E\left[  \exp\left(  \chi\left(  \theta\right)  \tau^{+}\left(  \mathbf{y}%
\right)  \right)  \right]  \leq\exp\left(  h\left(  \mathbf{y};\theta
,\varepsilon\right)  \right)  ,
\]
and $h\left(  \mathbf{y};\theta,\varepsilon\right)  +\chi\left(
\theta\right)  \rightarrow0$ as $\theta\rightarrow0$. It turns out that
$\chi\left(  \theta\right)  =O\left(  \theta/d\right)  $. Step 2.1 is executed
by means of a suitable Lyapunov argument in Lemma \ref{Lem_LB_CT}.

\textbf{Step 2.2 (geometric trials for visits to zero):} Step 2.1 allows us to
estimate the time until all of the components of the process $\mathbf{Y}%
\left(  \cdot\right)  $ are inside a compact set (this is due to the
domination property of $\mathbf{Y}^{+}$ and Assumption A2)). Then, using a
geometric trial argument, we estimate the time it takes for the $d$%
-coordinates of process $\mathbf{Y}$ to visit zero (i.e. when $\eta^{1}\left(
\mathbf{y}\right)  $, defined in Step 1, occurs). This estimate is somewhat
analogous to a coupon collection problem. In detail, suppose there are in total $d$ different types of coupons, the coupon collection problem deals with the number of coupons one needs to collect before obtaining all the coupon types (see page 61 of \cite{Feller_1968}). In our setting, the $i$-th type of coupon is collected when
the $i$-th coordinate, $Y_{i}$, visits zero and we want to estimate the time needed to collect all $d$ types of coupons. 

Assumptions A1) to A3) allow us to obtain  in Lemma \ref{Lem:Coupon} suitably uniform estimates on the
probability that a particular type of coupon is collected conditional on the event
that a given set of coupons has already been collected. But one has to keep
track of the coordinates of the upper bound process each time one attempts to
collect a new type of coupon. We do this by a stochastic domination argument. In Lemma \ref{Lemma_Stoch_Dom}, we obtain a coupling which implies the bound $\eta^{n}\left(
\mathbf{y}\right)  \leq\tau^{+}\left(  \mathbf{y}\right)  +\xi_{1}+...+\xi
_{n}$ where $\xi_{i}$'s are some i.i.d. positive random variables independent
of $\tau^{+}(\mathbf{y})$.
Then, we obtain  in Lemma \ref{Lem_LB_DT} a
bound of the following form:
\[
E\left[  \exp\left(  \chi\left(  \theta\right)  \tau^{+}\left(  \mathbf{y}%
\right)  +\chi\left(  \theta\right)  \xi_{1}\right)  \right]  \leq\exp\left(
h\left(  \mathbf{y};\theta,\varepsilon\right)  \right)  E\left[  \exp\left(
\chi\left(  \theta\right)  \xi_{1}\right)  \right]  .
\]

\textbf{Step 2.3 (connecting back to }$\mathcal{N}\left(  t;\mathbf{y}\right)
$\textbf{):} A standard supermartingale argument, using the domination
involving i.i.d. random variables, $\xi_{i}$'s, as discussed in Step 2.2, results
in the bound in Lemma \ref{Lem_LB_Nt_Bound},
\[
E\left(  1-\beta_{0}\right)  ^{\mathcal{N}\left(  t;\mathbf{y}\right)
}=O\left(  \exp\left(  h\left(  \mathbf{y};\theta,\varepsilon\right)
-\chi\left(  \theta\right)  t\right)  \right)  ,
\]
which holds uniformly in $d$ as $t\rightarrow\infty$ -- assuming that $\theta$
is suitably chosen as a function of $\beta_{0}$. It turns out that the
selection of $\theta$ forces $\chi\left(  \theta\right)  =O\left(
1/(d^{4}\log\left(  d\right)  )\right)  $.

\textbf{Step 3: }We conclude the result by putting all of the previous steps together.

\section{Step 1: Bounding the Difference of the Coupled
Processes\label{Section_Step_1}}

Here, we introduce an auxiliary Markov chain $\left(  M\left(  n\right)
:n\geq0\right)  $ living on the state space $\{0,1,...,d\}$ so that $P\left(
M\left(  n+1\right)  =j\mid M\left(  n\right)  =i\right)  =Q_{i,j}$ for $1\leq
i,j\leq d$. State $0$ is an absorbing state and $P\left(  M\left(  n+1\right)
=0\mid M\left(  n\right)  =i\right)  =Q_{i,0}=1-\sum_{j=1}^{d}Q_{i,j}$. We use
$P_{i}$ to refer to the probability law given that $M(0)=i$. For any subset
$S\subseteq\{1,...,d\}$, we define
\begin{align*}
\tau\left(  S\right)   &  =\inf\{n\geq0:M\left(  n\right)  \in S\},\text{
and}\\
\tau\left(  \{0\}\right)   &  =\inf\{n\geq0:M\left(  n\right)  =0\}.
\end{align*}
Define the $d\times d$ matrix $\Lambda\left(  S\right)  $ as
\[
\Lambda_{i,j}\left(  S\right)  =P_{i}\left(  \tau\left(  S\right)
<\tau\left(  \{0\}\right)  ,M\left(  \tau\left(  S\right)  \right)  =j\right)
\]
for $i,j\in\{1,...,d\}$.

\textbf{Remark:} The main purpose of introducing the auxiliary Markov chain $M(\cdot)$ in the proof is to provide a representation of the matrix $\Lambda(\cdot)$ as transition probabilities of $M$, such that we can derive the upper bound for the difference of the two coupled RBMs in terms of absorbing probabilities as in Lemma \ref{Lem_DB1}. Therefore, we are only interested in the distributional information of $M$ and do not need to couple it with the RBM in the same probability space.

\begin{lemma}
\label{lem: Lambda} The matrix $\Lambda(S)$ can be represented as
\[
\Lambda(S)=\left(
\begin{array}
[c]{cc}%
\Lambda_{SS} & \Lambda_{S\bar{S}}\\
\Lambda_{\bar{S}S} & \Lambda_{\bar{S}\bar{S}}%
\end{array}
\right)  =\left(
\begin{array}
[c]{cc}%
I & 0\\
-(R_{S\bar{S}}R_{\bar{S}\bar{S}}^{-1})^{T} & 0
\end{array}
\right)  .
\]
As a result,
\[
\Lambda^{T}(S)=\left(
\begin{array}
[c]{cc}%
I & -R_{S\bar{S}}R_{\bar{S}\bar{S}}^{-1}\\
0 & 0
\end{array}
\right)  .
\]

\end{lemma}

Recall that we have defined a sequence of stopping times $\eta_{i}%
^{k}(\mathbf{y})$ and $\eta^{k}\left(  \mathbf{y}\right)  $ in \eqref{eq: eta}
. Let%
\begin{align*}
\Gamma_{i}\left(  t,\mathbf{y}\right)  =\{\eta_{i}^{k}:\eta_{i}^{k}\leq t\},
\text{ and } \Gamma\left(  t,\mathbf{y}\right)  =\cup_{i=1}^{d}\Gamma
_{i}\left(  t,\mathbf{y}\right)  .
\end{align*}
For any time point $t\geq0$, define%
\[
\mathcal{C}\left(  t\right)  =\{1\leq i\leq d:Y_{i}\left(  t\right)
=0\}\text{ and }\mathcal{\bar{C}}\left(  t\right)  =\{1\leq j\leq
d:j\notin\mathcal{C}\left(  t\right)  \}
\]
We are ready to provide a bound for $\mathbf{1}^{T}(\mathbf{Y}(t;\mathbf{y}%
)-\mathbf{Y}(t;\mathbf{0}))$.

\begin{lemma}
\label{Lem:DM}
\[
0\leq\mathbf{1}^{T}(\mathbf{Y}(t;\mathbf{y})-\mathbf{Y}(t;\mathbf{0}%
))\leq\mathbf{1}^{T}\prod\limits_{s\in\Gamma\left(  t,\mathbf{y}\right)
}\Lambda^{T}\left(  \mathcal{\bar{C}}\left(  s\right)  \right)  \mathbf{y}.
\]

\end{lemma}

\bigskip

The proofs of Lemmas \ref{lem: Lambda} and \ref{Lem:DM} can be found at the
end of this section. Given Lemma \ref{Lem:DM}, we can provide an exponentially
decaying upper bound in terms of $\mathcal{N}\left(  t;\mathbf{y}\right)  $.
The intuition is that the matrices $\Lambda\left(  \mathcal{\bar{C}}\left(
t\right)  \right)  $ are substochastic and thus one might hope to obtain an
exponentially decaying bound.

\begin{lemma}
\label{Lem_DB1}%
\begin{align}
\label{Bnd_DB1}\mathbf{1}^{T}(\mathbf{Y}(t;\mathbf{y})-\mathbf{Y}%
(t;\mathbf{0}))  &  \leq\left\Vert \mathbf{y}\right\Vert _{1}\cdot d\kappa_{0}
\left(  1-\beta_{0}\right)  ^{\mathcal{N}\left(  t;\mathbf{y}\right)
}.\nonumber
\end{align}

\end{lemma}

\begin{proof}[Proof of Lemma \protect\ref{Lem_DB1}]
For any $k>0$, we write $\eta^k_{(1)}\leq \eta^k_{(2)}\leq ...\leq \eta^k_{(d)}$ as the sorting of $\{\eta^k_1,...,\eta^k_d\}$. Ties between $\eta^k_i$ and $\eta^k_j$ for $i \neq j$ are resolved arbitrarily, for example, lexicographically comparing $i$ and $j$. For the Markov chain $M(n)$, as we have defined at the beginning of this section, we define a sequence of stopping times $\tau_{j}^k$ as the following:
\begin{align*}
\tau_{1}^1&=\inf\{n\geq 0: M(n)\in \bar{\mathcal{C}}(\eta^1_{(1)})\},\\
\tau_{j+1}^k&=\inf\{n\geq \tau_{j}^k: M(n)\in \bar{\mathcal{C}}(\eta^k_{(j+1)})\}\text{ for all }j\leq d-1, \\
\tau_{1}^{k+1}&=\inf\{n\geq \tau_{d}^k: M(n)\in \bar{\mathcal{C}}(\eta^{k+1}_{(1)})\}.\end{align*}
Then, for any $m>0$ and $1\leq i\leq d$, one can check that
\begin{equation*}
\left(\prod_{k=1}^m\prod_{j=1}^d \Lambda(\bar{\mathcal{C}}(\eta^k_{(j)}))\mathbf{1}\right)_i=P_i(\tau_1^1\leq\tau_2^1\leq ....\leq \tau^m_d<\tau(\{0\})).\end{equation*}
We show that $\tau^m_d\geq m$ almost surely conditional on the event that $\tau_1^1\leq\tau_2^1\leq ....\leq \tau^m_d<\tau(\{0\})$. First,  we show  that $\tau^1_d\geq 1$. Suppose $\eta^1_{i}=\eta^1_{(j_1)}$,  then, since $i\notin \bar{\mathcal{C}}(\eta^1_{(j_1)})$ and $M(0)=i$, we must have $\tau^1_d\geq \tau^1_{j_1}\geq1$. For any $1\leq k\leq m$, let $l=M(\tau^k_d)$. Suppose $\eta^{k+1}_{l}=\eta^{k+1}_{(j_l)}$. Since $l\notin \bar{\mathcal{C}}(\eta^{k+1}_{(j_l)})$ and $M(\tau^k_d)=l$, we must have that $\tau^{k+1}_d\geq \tau^{k+1}_{j_l}\geq \tau^k_d+1$. Therefore, we can conclude by induction that $\tau^m_d\geq m$, and hence $\tau(\{0\})\geq m$ conditional on the event that $\tau_1^1\leq\tau_2^1\leq ....\leq \tau^m_d<\tau(\{0\})$. As a result, we have
\begin{equation*}
\left(\prod_{k=1}^m\prod_{j=1}^d \Lambda(\bar{\mathcal{C}}(\eta^k_{(j)}))\mathbf{1}\right)_i\leq P_i(\tau(\{0\})\geq m).
\end{equation*}
As $Q_{i,j}=P\left(M\left( n+1\right) =j\mid M\left( n\right) =i\right)$ for $1\leq i, j\leq d$ and $0$ is the absorbing state,
\begin{equation*}
\max_{i}P_{i}\left( \tau (\{0\})>n\right) =\left\Vert Q^{n}\mathbf{1}%
\right\Vert _{\infty }.
\end{equation*}%
Under Assumption A1),
\begin{equation*}
\left\Vert Q^{n}\mathbf{1}\right\Vert _{\infty }\mathbf{\leq 1}^{T}Q^{n}%
\mathbf{1}\leq d\left\Vert \mathbf{1}^{T}Q^{n}\right\Vert _{\infty }\leq
d\kappa _{0}\left( 1-\beta _{0}\right) ^{n}.\end{equation*}%
As a result, we have
$$\left\Vert\prod_{k=1}^m\prod_{j=1}^d\Lambda(\bar{\mathcal{C}}(\eta^k_{(j)}))\mathbf{1}\right\Vert_\infty
\leq d\kappa _{0}
\left( 1-\beta _{0}\right) ^{m}.$$
Let $t_1=\eta^{\mathcal{N}(t;\mathbf{y})}$ and recall from the definition of $\mathcal{N}(t;\mathbf{y})$ that $t_1\leq t$. Then, we have
\begin{align*}
\mathbf{1}^{T}(\mathbf{Y}(t;\mathbf{y})-\mathbf{Y}(t;\mathbf{0}))&\leq \mathbf{1}^{T}(\mathbf{Y}(t_1;\mathbf{y})-\mathbf{Y}(t_1;\mathbf{0}))\\
&\leq \mathbf{1}^T\prod_{k=1}^{\mathcal{N}\left( t;\mathbf{y}\right)}\prod_{j=1}^d \Lambda^T(\bar{\mathcal{C}}(\eta^k_{(j)}))\mathbf{y}\\
&\leq \left\Vert \prod_{k=1}^{\mathcal{N}\left( t;\mathbf{y}\right)}\prod_{j=1}^d \Lambda(\bar{\mathcal{C}}(\eta^k_{(j)}))\mathbf{1} \right\Vert_\infty \|\mathbf{y}\|_1\\
&\leq \|\mathbf{y}\|_1d\kappa _{0}
\left( 1-\beta _{0}\right) ^{\mathcal{N}\left( t;\mathbf{y}\right)}.
\end{align*}
Here, the first inequality follows part (iv) of Theorem 1 of \cite{KellaRamasubramanian_2012}, which states that for two RBMs that are only different in their initial values, their difference is non-increasing in time. The second inequality follows Lemma \ref{Lem:DM}.
\end{proof}

\begin{proof}[Proof of Lemma \ref{lem: Lambda}]
Following the definition of the matrix $\Lambda(S)$, it is obvious that, for all $j\in \bar{S}$,
$$\Lambda_{i,j}(S)=0\text{ as }P(M(\tau(S))=j)=0,$$
and for all $i, j \in S$,
$$\Lambda_{i,j}(S)=\delta_{i,j} \text{ as }\tau(S)=0\text{ and }M(\tau(S))=i.$$
Therefore, $\Lambda_{SS}=I$ and all elements of $\Lambda_{S\bar{S}}$ and $\Lambda_{\bar{S}\bar{S}}$ are 0.
By the property of Markov chains with transient states, we can compute that
\begin{align*}
\Lambda_{\bar{S}S}&=Q_{\bar{S}S}+Q_{\bar{S}\bar{S}}Q_{\bar{S}{S}}+Q_{\bar{S}\bar{S}}^2Q_{\bar{S}{S}}+.....\\
&=(I+Q_{\bar{S}\bar{S}}+Q_{\bar{S}\bar{S}}^2+...)Q_{\bar{S}{S}}=(I-Q_{\bar{S}\bar{S}})^{-1}Q_{\bar{S}{S}}
\end{align*}
Note that $R=(I-Q)^T$. As a result,  we have that $(I-Q_{\bar{S}\bar{S}})=R_{\bar{S}\bar{S}}^T$ and $Q_{\bar{S}S}=-R_{S\bar{S}}^T$, and therefore $\Lambda_{\bar{S}S}=-(R_{S\bar{S}}R^{-1}_{\bar{S}\bar{S}})^T$.
\end{proof}

\begin{proof}[Proof of Lemma \protect\ref{Lem:DM}]
For simplicity of notation, we write $\tilde{\mathbf{Y}}(t)=\mathbf{Y}(t;%
\mathbf{y})$ and $\mathbf{Y}(t)=\mathbf{Y}(t;\mathbf{0})$. Since $\Gamma (t,\mathbf{y})$ is
a finite set for all $t$, let $t_{1}$ be the maximum of set $\Gamma (t,\mathbf{y})$ and
denote $\mathcal{C}=\mathcal{C}(t_{1})$. If $\Gamma(t,\mathbf{y})$ is empty, we define $t_1=0$.
We will prove the following statement:
\begin{equation}  \label{Eq:lm1}
\tilde{\mathbf{Y}}(t_{1})-\mathbf{Y}(t_{1})\leq \prod\limits_{s\in \Gamma
\left( t,\mathbf{y}\right) }\Lambda ^{T}\left( \mathcal{\bar{C}}\left(
s\right) \right) \mathbf{y}+H\mathbf{w,}
\end{equation}%
for some $\mathbf{w}\geq 0$ and $H$ is a matrix defined via
\begin{equation*}
H_{ij}=1(i\in \bar{\mathcal{C}},j\in \bar{\mathcal{C}})\cdot (P_{j}(\tau(\{i\})<\tau (\{0\})\text{
and }W(n)\in \mathcal{C}\text{ for all }n\leq \tau(\{i\})-1)-\delta _{ij}).
\end{equation*}%
Then, we can conclude
\begin{equation*}
\mathbf{1}^{T}(\tilde{\mathbf{Y}}(t)-\mathbf{Y}(t))\leq \mathbf{1}^{T}(%
\tilde{\mathbf{Y}}(t_{1})-\mathbf{Y}(t_{1}))\leq \mathbf{1}%
^{T}\prod\limits_{s\in \Gamma \left( t,\mathbf{y}\right) }\Lambda ^{T}\left(
\mathcal{\bar{\mathcal{C}}}\left( s\right) \right) \mathbf{y},
\end{equation*}%
where the first inequality holds following Part (iv) of Theorem 1 in \cite{KellaRamasubramanian_2012} and the  last  holds as $\mathbf{1}^{T}H\leq 0$.
Now, we shall prove \eqref{Eq:lm1} by induction on the cardinality of $\Gamma \left( t,\mathbf{y}%
\right) $. The base case is that $\Gamma(t,\mathbf{y})$ is empty. Then, for any $t$, as long as $\Gamma(t,\mathbf{y})$ is empty, $t_1=0$ and hence
$$\tilde{\mathbf{Y}}(t_{1})-\mathbf{Y}(t_{1})=\tilde{\mathbf{Y}}(0)-\mathbf{Y}(0)=\mathbf{y}$$
and \eqref{Eq:lm1} holds for $\mathbf{w}=0$.
Suppose \eqref{Eq:lm1} holds for all $t$ such that the cardinality of $\Gamma(t,\mathbf{y})\leq k$. Consider the case that $\Gamma(t,\mathbf{y})=k+1$. Let $t_{2}$ be the second largest element of the set $\Gamma (t,\mathbf{y%
})$.
Let $\mathbf{z}=\tilde{\mathbf{Y}}(t_{2})-%
\mathbf{Y}(t_{2})$ and $\mathbf{w}=(\mathbf{L}(t_{1})-\mathbf{L}(t_{2}))-(%
\tilde{\mathbf{L}}(t_{1})-\tilde{\mathbf{L}}(t_{2}))\geq 0$ (see Theorem 1
in \cite{KellaRamasubramanian_2012}). At time $t_{1}$, by definition, we
have
\begin{equation*}
\tilde{\mathbf{Y}}(t_{1})-\mathbf{Y}(t_{1})=\mathbf{z}-R\mathbf{w}.
\end{equation*}%
As $\tilde{\mathbf{Y}}_{\mathcal{C}}(t_{1})=\mathbf{Y}_{\mathcal{C}}(t_{1})=0$,
\begin{equation*}
0=\tilde{\mathbf{Y}}_{\mathcal{C}}(t_{1})-\mathbf{Y}_{\mathcal{C}}(t_{1})=\mathbf{z}_{\mathcal{C}}-R_{\mathcal{CC}}%
\mathbf{w}_{\mathcal{C}}-R_{\mathcal{C}\bar{\mathcal{C}}}\mathbf{w}_{\bar{\mathcal{C}}},
\end{equation*}%
from which we solve $\mathbf{w}_{\mathcal{C}}=R_{\mathcal{CC}}^{-1}(\mathbf{z}_{\mathcal{C}}-R_{\mathcal{C}\bar{\mathcal{C}}}%
\mathbf{w}_{\bar{\mathcal{C}}})$. Therefore,
\begin{align*}
\tilde{\mathbf{Y}}_{\bar{\mathcal{C}}}(t_{1})-\mathbf{Y}_{\bar{\mathcal{C}}}(t_{1})& =\mathbf{z}%
_{\bar{\mathcal{C}}}-R_{\bar{\mathcal{C}}\mathcal{C}}\mathbf{w}_{\mathcal{C}}-R_{\bar{\mathcal{C}}\bar{\mathcal{C}}}\mathbf{w}_{\bar{\mathcal{C}}}
\\
& =\mathbf{z}_{\bar{\mathcal{C}}}-R_{\bar{\mathcal{C}}\mathcal{C}}R_{\mathcal{CC}}^{-1}(\mathbf{z}_{\mathcal{C}}-R_{\mathcal{C}\bar{\mathcal{C}}}%
\mathbf{w}_{\bar{\mathcal{C}}})-R_{\bar{\mathcal{C}}\bar{\mathcal{C}}}\mathbf{w}_{\bar{\mathcal{C}}} \\
& =(I\mathbf{z}_{\bar{\mathcal{C}}}-R_{\bar{\mathcal{C}}\mathcal{C}}R_{\mathcal{CC}}^{-1}\mathbf{z}_{\mathcal{C}})+(R_{\bar{\mathcal{C}}%
\mathcal{C}}R_{\mathcal{CC}}^{-1}R_{\mathcal{C}\bar{\mathcal{C}}}-R_{\bar{\mathcal{C}}\bar{\mathcal{C}}})\mathbf{w}_{\bar{\mathcal{C}}} \\
& =\Lambda _{\bar{\mathcal{C}}}^{T}(\bar{\mathcal{C}})\mathbf{z}+(R_{\bar{\mathcal{C}}\mathcal{C}}R_{\mathcal{CC}}^{-1}R_{\mathcal{C}\bar{\mathcal{C}}%
}-R_{\bar{\mathcal{C}}\bar{\mathcal{C}}})\mathbf{w}_{\bar{\mathcal{C}}},
\end{align*}%
where the last equation holds following Lemma \ref{lem: Lambda}.
Note that
\begin{equation*}
R_{\bar{\mathcal{C}}\mathcal{C}}R_{\mathcal{CC}}^{-1}R_{\mathcal{C}\bar{\mathcal{C}}}-R_{\bar{\mathcal{C}}\bar{\mathcal{C}}}=Q_{\bar{\mathcal{C}}%
\mathcal{C}}^{T}(I-Q_{\mathcal{CC}}^{T})^{-1}Q_{\mathcal{C}\bar{\mathcal{C}}}^{T}+Q_{\bar{\mathcal{C}}\bar{\mathcal{C}}}^{T}-I_{\bar{\mathcal{C}}},
\end{equation*}%
where $Q$ is the transition matrix of $M$. Let $H_{\bar{\mathcal{C}}\bar{\mathcal{C}}}=R_{\bar{\mathcal{C}}%
\mathcal{C}}R_{\mathcal{CC}}^{-1}R_{\mathcal{C}\bar{\mathcal{C}}}-R_{\bar{\mathcal{C}}\bar{\mathcal{C}}}$. From the definition of $Q$,
we can check that
\begin{equation*}
H_{ij}=(P_{j}(\tau _{i}<\tau (\{0\})\text{ and }M(n)\in \mathcal{C}\text{ for all }%
n\leq \tau _{i}-1)-\delta _{ij}),
\end{equation*}%
for all $i,j\in \bar{\mathcal{C}}$ and $\tau _{i}:=\inf \{n\geq 1:M(n)=i\}$. Note that
$\Lambda _{\mathcal{C}}(\bar{\mathcal{C}})=0$ following Lemma \ref{lem: Lambda}, so we have
\begin{equation*}
\tilde{\mathbf{Y}}(t_{1})-\mathbf{Y}(t_{1})=\Lambda ^{T}(\mathcal{C})\mathbf{z}+H%
\mathbf{w}.
\end{equation*}%
Note that the cardinality of $\Gamma(t_2, \mathbf{y})= k$ and $t_2$ is its maximum, so by induction, we have
\begin{equation*}
\mathbf{z}\leq \prod\limits_{s\in \Gamma \left( t,\mathbf{y}\right)
\setminus \{t_{1}\}}\Lambda ^{T}\left( \mathcal{\bar{\mathcal{C}}}\left( s\right)
\right) \mathbf{y}+H^{\ast }\mathbf{w}^{\ast },
\end{equation*}%
where $\mathbf{w}^{\ast }\geq 0$ and
\begin{equation*}
H_{ij}^{\ast }=1(i\in \bar{D},j\in \bar{D})(P_{j}(\tau _{i}<\tau (\{0\})%
\text{ and }M(n)\in D\text{ for all }n\leq \tau _{i}-1)-\delta _{ij}),
\end{equation*}%
with $D=\mathcal{C}(t_2)$. As $\Lambda (\bar{\mathcal{C}})\geq 0$, so we have
\begin{equation*}
\tilde{\mathbf{Y}}(t_{1})-\mathbf{Y}(t_{1})\leq \prod\limits_{s\in \Gamma
\left( t,\mathbf{y}\right) }\Lambda ^{T}\left( \mathcal{\bar{\mathcal{C}}}\left(
s\right) \right) \mathbf{y}+\Lambda ^{T}(\bar{\mathcal{C}})H^{\ast }\mathbf{w}^{\ast }+H%
\mathbf{w}.
\end{equation*}%
As $\mathbf{w}^{\ast }\geq 0$, it suffices to show
that $(\Lambda ^{T}(\bar{\mathcal{C}})H^{\ast })_{ij}\leq 0$ for all $1\leq i, j\leq d$. Note that
$$(\Lambda ^{T}(\bar{\mathcal{C}})H^{\ast })_{ij}=\sum_{k}\Lambda ^{T}(\bar{\mathcal{C}})_{ik}H_{kj}^{\ast }.$$
Since $H_{kj}^{\ast }=0$ for all $j\in D$, we conclude that $(\Lambda ^{T}(\bar{\mathcal{C}})H^{\ast })_{ij}=0$ for all $j\in D$.
For $j\in \bar{D}$,
recall that $\Lambda ^{T}(\bar{\mathcal{C}})_{ij}=P_{j}(\tau \left( \bar{\mathcal{C}}\right) <\tau
\left( \{0\}\right) ,M\left( \tau (\bar{\mathcal{C}})\right) =i)$, therefore
\begin{align}\label{eq: H}
& ~(\Lambda ^{T}(\mathcal{\bar{C}})H^{\ast })_{ij}=\sum_{k}\Lambda ^{T}(\mathcal{\bar{C}})_{ik}H_{kj}^{\ast }
\nonumber\\
=& \sum_{k\in \bar{D}}P_{k}(\tau \left( \bar{\mathcal{C}}\right) <\tau \left(
\{0\}\right) ,M\left( \tau (\bar{\mathcal{C}})\right) =i)(P_{j}(\tau _{k}<\tau (\{0\})%
\text{ and }M(n)\in D\text{ for all }n\leq \tau _{k}-1)-\delta _{kj}) \nonumber\\
=& P_{j}(\tau(\bar{D})<\tau(\{0\}), \tilde{\tau}({\bar{\mathcal{C}}})<\tau(\{0\}), M(n)\in D\text{ for all }n< \tau ({\bar{D}}),M\left(
\tau (\bar{\mathcal{C}})\right) =i) \nonumber\\
&~~ -P_{j}(\tau(\bar{\mathcal{C}})<\tau(\{0\}),M\left( \tau (\bar{\mathcal{C}})\right) =i) \\
\leq & 0,\nonumber
\end{align}
where $\tilde{\tau}({\bar{\mathcal{C}}})\doteq \inf\{t\geq \tau(\bar{D}):M(t)\in\bar{\mathcal{C}}\}$ and the inequality holds as the first probability event is a subset of the latter one in \eqref{eq: H}.
\end{proof}

\section{Step 2: Coupling, Lyapunov Bounds, and Geometric
Trials\label{Subsect_tech_pfs}}

The main result in this section is the following.

\begin{proposition}
\label{prop:A2}Under A1) to A3), for any $\beta>0\ $satisfying (\ref{Geo_1}),
we have%
\[
E\left[  \left(  1-\beta\right)  ^{\mathcal{N}\left(  t;\mathbf{y}\right)
}\right]  \leq\exp\left(  \zeta_{0}\left\Vert \mathbf{y}\right\Vert _{\infty
}/(d^{3}\log\left(  d\right)  )+\beta/d^{2}\right)  \cdot\exp\left(
-\zeta_{1}t/(d^{4}\log\left(  d\right)  )\right)  \cdot\left(  1-\beta\right)
^{-1},
\]
where
\[
\zeta_{0}=\frac{\delta_{1}\cdot\beta^{2} }{2\max_{i=1}^{d}\sigma_{i}^{2}%
},\text{ \ }\zeta_{1}=\frac{\delta_{1}^{2}\cdot\beta^{2} }{16\max_{i=1}%
^{d}\sigma_{i}^{2}}.
\]

\end{proposition}

The proof of Proposition \ref{prop:A2} follows Steps 2.1, 2.2 and 2.3 as
described in the main strategy. The proofs of all the technical lemmas can be
found in Section \ref{Sec:proofs}.

We first explain how to construct the upper bound process $\mathbf{Y}%
^{+}(\cdot;\mathbf{y})$ briefly mentioned in the discussion of Step 2.
Following Assumptions A1) and A3), $\left\Vert R^{-1}\mathbf{1}\right\Vert
_{\infty}\leq\kappa_{0}/\beta_{0}$, and $R^{-1}\boldsymbol{\mu}\leq-\delta
_{0}\mathbf{1}$. We choose
\[
\boldsymbol{\mu}^{+}=\boldsymbol{\mu}+\delta_{1}\mathbf{1,}%
\]
where $\delta_{1}=\delta_{0}\beta_{0}/(2\kappa_{0})$. One can check that
$\boldsymbol{\mu}^{+}>\boldsymbol{\mu}$ and $R^{-1}\boldsymbol{\mu}^{+}%
\leq-(\delta_{0}/2)\mathbf{1}$.

Let ($\mathbf{Y}^{+}(\cdot)$, $\mathbf{L}^{+}(\cdot)$) be the solution to the
Skorokhod problem with orthogonal reflection as follows,%
\[
\mathbf{Y}^{+}\left(  t\right)  =\mathbf{Y}^{+}(0)+\mathbf{\bar{X}}\left(
t\right)  +\mathbf{L}^{+}\left(  t\right)  ,
\]
with $\mathbf{\bar{X}}\left(  t\right)  =\mathbf{X}\left(  t\right)
-\boldsymbol{\mu}^{+}t$ and $\mathbf{Y}^{+}\left(  0\right)  =\mathbf{y}$. We
write $\mathbf{Y}^{+}(t)$ as $\mathbf{Y}^{+}(t;\mathbf{y})$, as its value
depends on the initial value $\mathbf{y}$.
We know from Lemma 3.1 in \cite{KellaWhitt_1996} that
\begin{equation}
R^{-1}\mathbf{Y}\left(  t;\mathbf{y}\right)  \leq R^{-1}\mathbf{Y}^{+}\left(
t;\mathbf{y}\right)  . \label{SD_KW_96}%
\end{equation}

As discussed in Step 2.1, we have defined $\tau^{+}\left(  \mathbf{y}\right)
=\inf\{t\geq0:\mathbf{Y}^{+}\left(  t;\mathbf{y}\right)  \leq\mathbf{1}\}$,
which is the time to visit a compact set for $\mathbf{Y}^{+}$, and for
$\mathbf{Y}$ as well, according to \eqref{SD_KW_96}:
\[
\mathbf{Y}\left(  \tau^{+}\left(  \mathbf{y}\right)  ;\mathbf{y} \right)  \leq
R^{-1}\mathbf{Y}\left(  \tau^{+}\left(  \mathbf{y}\right)  ;\mathbf{y}
\right)  \leq R^{-1}\mathbf{Y}^{+}\left(  \tau^{+}\left(  \mathbf{y}\right)
;\mathbf{y} \right)  \leq R^{-1}\mathbf{1}\leq\frac{\kappa_{0}}{\beta_{0}%
}\mathbf{1=}b_{1}\mathbf{1},
\]
where the first inequality holds as $R^{-1}\geq I$ and $\mathbf{Y}\geq0$.
(Note that $R^{-1}=(I-Q)^{-1}=I+Q+Q^{2}+...\geq I$ as $Q\geq0$. ) The
following result provides a bound for the moment-generating function of
$\tau^{+}\left(  \mathbf{y}\right)  $.

\begin{lemma}
\label{Lem_LB_CT}Define $g:[0,\infty)\to\mathbb{R}_{+}$ as
\[
g\left(  y\right)  =2^{-1}y^{2}I\left(  0\leq y\leq1\right)  +(y-1/2)I\left(
y>1\right)  .
\]
%
For any given $\varepsilon>0$ and $\theta>0$, define
\[
h\left(  \mathbf{y;}\theta,\varepsilon\right)  =\varepsilon\log\left(
\sum_{i=1}^{d}\exp\left(  g(\theta y_{i}) /\varepsilon\right)  \right)
\leq\max_{i=1}^{d}g( \theta y_{i}) +\varepsilon\log\left(  d\right)
\leq\theta\left\Vert \mathbf{y}\right\Vert _{\infty}+\varepsilon\log\left(
d\right)  .
\]
Then, for any
\[
0<\theta\leq\frac{\varepsilon}{2\varepsilon+1}\cdot\frac{\delta_{1}}%
{(1+d)\max_{i=1}^{d}\sigma_{i}^{2}}\leq\frac{\varepsilon}{2\varepsilon+1}%
\cdot\frac{\delta_{0}\beta_{0}/(2\kappa_{0})}{\left(  1+d\right)  b_{0}},
\]
and
\begin{equation}
\chi\left(  \theta\right)  \doteq\theta\frac{\delta_{1}}{2(1+d)}\leq
\theta\frac{\delta_{0}\beta_{0}/(2\kappa_{0})}{2(1+d)}, \label{Select_chi}%
\end{equation}
we have%
\begin{equation}
E\left[  \exp\left(  h\left(  \mathbf{Y}^{+}\left(  \tau^{+}\left(
\mathbf{y}\right)  \right)  \mathbf{;}\theta,\varepsilon\right)  +\chi\left(
\theta\right)  \tau^{+}\left(  \mathbf{y}\right)  \right)  \right]  \leq
\exp\left(  h\left(  \mathbf{y;}\theta,\varepsilon\right)  \right)  \leq
\exp\left(  \theta\left\Vert \mathbf{y}\right\Vert _{\infty}+\varepsilon
\log\left(  d\right)  \right)  . \label{in_0_Mg}%
\end{equation}

\end{lemma}

\bigskip


Starting from position $\mathbf{Y}\left(  \tau^{+}\left(  \mathbf{y}\right)
\right)  $, we wait for another unit of time till $\tau^{+}(\mathbf{y})+1$. If
the event $\{Y_{i}(t)=0\text{ for some }\tau^{+}(\mathbf{y})< t\leq\tau
^{+}(\mathbf{y})+1\}$ occurs, then we can conclude that $\eta^{1}_{i}\leq
\tau^{+}(\mathbf{y})+1$. The following lemma shows that, for all $1\leq i\leq
d$, the probability for such an event to happen is uniformly bounded away from
0, regardless of the position of the process at time $\tau^{+}(\mathbf{y})$.

\begin{lemma}
\label{Lem:Coupon} There exists a constant $p_{0}>0$, independent of $d$, such
that for all $\mathbf{y}\leq b_{1}\mathbf{1}$ and every $i\in\{1,...,d\} $,%
\[
P(Y_{i}(t)=0\text{ for some }t\leq1|\mathbf{Y}(0)=\mathbf{y})\geq p_{0}.
\]

\end{lemma}

Based on Lemma \ref{Lem:Coupon}, we are ready to perform a \textquotedblleft
geometric trial argument\textquotedblright(Step 2.2) to obtain a bound for
each $\eta^{1}_{i}$ with $1\leq i\leq d$. Each round of the trials includes
two steps described as follows. Suppose at the beginning of the $k$-th round
of trial, the initial position of the process $\mathbf{Y}$ is $\mathbf{Y}%
^{i,k}$ (in particular, $\mathbf{Y}^{i,1}=\mathbf{y}$). In the first step, it
takes $\tau^{+}(k;\mathbf{Y}^{i,k})$ for $\mathbf{Y}(\cdot; \mathbf{Y}^{i,k})$
to arrive to the compact set $\{\mathbf{y}\in\mathbb{R}^{d}: |y_{i}|\leq b_{1}
\text{ for }1\leq i\leq d\}$. (For given $\mathbf{y}$, $\tau^{+}%
(k;\mathbf{y})$'s are i.i.d. copies of $\tau^{+}(\mathbf{y})$.) Then, in the
next one unit of time, we check if the event $\{Y_{i}(t;\mathbf{Y}%
^{i,k})=0\text{ for some }\tau^{+}(k;\mathbf{Y}^{i,k})< t\leq\tau
^{+}(k;\mathbf{Y}^{i,k})+1\}$ happens. If so, we can stop as the process has
already hit 0. If not, we then start the next round of trial with the initial
position $\mathbf{Y}^{i,k+1}=\mathbf{Y}(\tau^{+}(k;\mathbf{Y}^{i,k}%
)+1;\mathbf{Y}^{i,k})$. In summary, we can define a sequence of Bernoulli
random variables $\zeta_{k}(i)$ jointly with the sequence $\{\mathbf{Y}%
^{i,k}\}$ as
\[
\zeta_{k}(i)\doteq1(Y_{i}(t;\mathbf{Y}^{i,k})=0\text{ for some }\tau
^{+}(k;\mathbf{Y}^{i,k})< t\leq\tau^{+}(k;\mathbf{Y}^{i,k})+1).
\]
Let $K=\min\{k: \zeta_{k}(i)=1\}$, and we obtain a bound for $\eta^{1}%
_{i}(\mathbf{y})$:
\[
\eta^{1}_{i}(\mathbf{y})\leq\sum_{k=1}^{K} (\tau^{+}(k;\mathbf{Y}^{i,k})+1).
\]
The next lemma shows that we can replace $K$ with a Geometric random variable
(r.v.) $G^{i}$, and the sequence $\tau^{+}(\mathbf{Y}^{i,k})$ with an i.i.d.
sequence of positive r.v.'s that are independent of $G$ and have bounded
moment-generating function.

\begin{lemma}
\label{Lemma_Stoch_Dom} Let $p$ be any positive number such that $p<p_{0}$.
Let $\{G^{i}:1\leq i\leq d\}$ be i.i.d. copies of a Geometric random variable
$G$ with probability of success equal to $p$. Then, we can construct a random
variable $\Theta_{d}>0$ and its i.i.d. copies $\{\Theta_{d}^{i,k}\}$ such
that
\[
\eta^{1}_{i}(\mathbf{y})\leq\tau^{+}(\mathbf{y})+\sum_{k=1}^{G^{i}}\left(
1+\tau^{+}(\Theta_{d}^{i,k}\mathbf{1})\right)  .
\]
Therefore,
\[
\eta^{1}(\mathbf{y})\leq\tau^{+}(\mathbf{y})+\sum_{i=1}^{d}\sum_{k=1}^{G^{i}}
\left(  1+\tau^{+}(\Theta_{d}^{i,k}\mathbf{1})\right)  .
\]
Moreover, let $\phi_{d}\left(  \theta\right)  =E[\exp\left(  \theta\Theta
_{d}\right)  ]$, then, for $\theta=o( 1) $ as $d\to\infty$,
\begin{equation}
\phi_{d}\left(  \theta\right)  \leq1+2(1-p)^{-1}\theta\log\left(  1+d\right)
\exp\left(  \theta\log\left(  1+d\right)  ^{2/3}\right)  +\theta O\left(
\exp\left(  -\log\left(  1+d\right)  ^{4/3}/3b_{0}\right)  \right)  .
\label{Bnd_fhi_d}%
\end{equation}

\end{lemma}

Define a random variable
\[
\xi=\sum_{i=1}^{d}\sum_{k=1}^{G^{i}}\left(  1+\tau^{+}(\Theta_{d}%
^{i,k}\mathbf{1})\right)  .
\]
According to Lemma \ref{Lemma_Stoch_Dom}, we can couple $\eta^{1}\left(
\mathbf{y}\right)  $ and $\xi$ so that%
\[
\eta^{1}\left(  \mathbf{y}\right)  \leq\tau^{+}\left(  \mathbf{y}\right)
+\xi,
\]
where $\tau^{+}\left(  \mathbf{y}\right)  $ is independent of $\xi$. The
Skorokhod problem is monotone with respect to the initial condition, i.e.
$\eta^{1}\left(  \mathbf{y}\right)  \leq\eta^{1}\left(  \mathbf{y}^{\prime
}\right)  $ whenever $\mathbf{y}\leq\mathbf{y}^{\prime}$ (see Theorem 1.1 (i)
of \cite{KellaRamasubramanian_2012}). As a result, we can iteratively apply
the previous reasoning. In particular, let $\xi_{1},\xi_{2},...$ be i.i.d.
copies of $\xi$ and independent of $\tau^{+}\left(  \mathbf{y}\right)  $.
Then, we can construct a coupling so that
\begin{align}
\eta^{1}\left(  \mathbf{y}\right)   &  \leq\tau^{+}\left(  \mathbf{y}\right)
+\xi_{1},\label{Stoch_Dom}\\
\eta^{2}\left(  \mathbf{y}\right)   &  \leq\tau^{+}\left(  \mathbf{y}\right)
+\xi_{1}+\xi_{2},\nonumber\\
&  ...\nonumber\\
\eta^{n}\left(  \mathbf{y}\right)   &  \leq\tau^{+}\left(  \mathbf{y}\right)
+\xi_{1}+...+\xi_{n}.\nonumber
\end{align}
Based on the bound of the moment-generating function of $\tau^{+}(\mathbf{y})$
in Lemma \ref{Lem_LB_CT} and Lemma \ref{Lemma_Stoch_Dom}, we have the
following result on the moment-generating function of $\eta^{n}(\mathbf{y})$
for all $n\geq1$.

\begin{lemma}
\label{Lem_LB_DT}For $n\geq1$,%
\[
E\exp\left(  \chi\left(  \theta\right)  \eta^{n}\left(  \mathbf{y}\right)
\right)  \leq\exp\left(  h\left(  \mathbf{y};\theta,\varepsilon\right)
\right)  \left(  \frac{\phi_{d}\left(  \theta\right)  \exp\left(  \chi(\theta)
+\varepsilon\log\left(  d\right)  \right)  p}{1-(1-p)\phi_{d}\left(
\theta\right)  \exp\left(  \chi(\theta) +\varepsilon\log\left(  d\right)
\right)  }\right)  ^{nd}.
\]
Moreover, suppose that $\varepsilon,$ $\theta>0$ are chosen so that%
\begin{equation}
\phi_{d}\left(  \theta\right)  \exp\left(  \chi\left(  \theta\right)
+\varepsilon\log\left(  d\right)  \right)  \leq\frac{1}{\text{ }(1-p)\left(
1+p\right)  }. \label{Select_Theta}%
\end{equation}
Then,
\[
E\exp\left(  \chi\left(  \theta\right)  \eta^{n}\left(  \mathbf{y}\right)
\right)  \leq\exp\left(  h\left(  \mathbf{y};\theta,\varepsilon\right)
\right)  \left(  1-p\right)  ^{-nd}.
\]

\end{lemma}

\bigskip

Finally, we obtain the following lemma, which takes us very close to the proof
of Proposition \ref{prop:A2}.

\begin{lemma}
\label{Lem_LB_Nt_Bound}Assume that $p=\min(p_{0}, \beta/d)$, $\varepsilon$ and
$\theta>0$ satisfies (\ref{Select_Theta}) and $\chi\left(  \theta\right)  $ is
defined according to (\ref{Select_chi}), we obtain that
\[
E\left(  \left(  1-p\right)  ^{d\cdot\mathcal{N}\left(  t;\mathbf{y}\right)
}\right)  \leq\exp\left(  h\left(  \mathbf{y;}\theta,\varepsilon\right)
\right)  \cdot\exp\left(  -\chi\left(  \theta\right)  t\right)  \cdot\left(
1-p\right)  ^{-d}.
\]

\end{lemma}

\bigskip

We now have all the ingredients required to provide a the proof of Proposition
\ref{prop:A2}.

\begin{proof}[Proof of Proposition \protect\ref{prop:A2}]
By Lemma \ref{Lem_LB_Nt_Bound}, the only step that remains is to select $\theta$, $\varepsilon$ satisfying \eqref{Select_Theta} and to estimate
the behavior of $\chi \left( \theta \right) $ assuming our selection of $p$ in Lemma  \ref{Lem_LB_Nt_Bound}. Given that $p=\min(p_0,\beta/d)$, we have
\begin{equation*}
\left( 1-p\right) ^{d}\geq \left( 1-\beta \right) .
\end{equation*}%
We then choose $\varepsilon $, $\theta $ as follows:
\begin{eqnarray*}
\varepsilon  &= &\frac{\beta^2 }{2d^{2}\log \left( d\right) }\text{, } \\
\theta  &= &\frac{\varepsilon }{2\varepsilon +1}\cdot \frac{\delta _{1}}{%
(1+d)\max_{i=1}^{d}\sigma _{i}^{2}}\leq \frac{\delta _{1}\cdot \beta^2 }{%
2d^{3}\log \left( d\right) \max_{i=1}^{d}\sigma _{i}^{2}}, \\
\end{eqnarray*}%
and hence
$$\chi \left( \theta \right) =\theta \frac{\delta _{1}}{2(1+d)}\leq\frac{%
\delta _{1}^{2}\cdot \beta^2 }{4d^{4}\log \left( d\right) \max_{i=1}^{d}\sigma
_{i}^{2}}.
$$
Therefore, for $d$ sufficiently large,
\begin{eqnarray*}
&&\phi _{d}\left( \theta \right) \exp \left( \chi \left( \theta \right)
+\varepsilon \log \left( d\right) \right)\\
&\leq &\exp \left( \chi(\theta)+\varepsilon \log \left( d\right) \right) \left( 1+2(1-p)^{-1}\theta \log \left( 1+d\right) \exp \left( \theta \log \left( 2+d\right)
^{2/3}\right) +\frac{\varepsilon}{4}\right)  \\
&\leq& \exp\left(\frac{%
\delta _{1}^{2}\cdot \beta^2 }{4d^{4}\log \left( d\right) \max_{i=1}^{d}\sigma
_{i}^{2}}+\frac{\beta^2}{2d^2}\right) \left( 1+2(1-p)^{-1}\theta \log \left( 1+d\right) \exp \left( \theta \log \left( 1+d\right)
^{2/3}\right)+\frac{\varepsilon}{4} \right) ,
\end{eqnarray*}
where the first inequality follows from (\ref{Bnd_fhi_d}) and the fact that the big-O term in \eqref{Bnd_fhi_d} goes to 0 as $d\to\infty$. Given our choice of $\theta$, we have
$$\exp\left(\frac{%
\delta _{1}^{2}\cdot \beta^2 }{4d^{4}\log \left( d\right) \max_{i=1}^{d}\sigma
_{i}^{2}}\right)=1+o(d^{-2}),\text{ and }\theta\log \left( 1+d\right) \exp \left( \theta \log \left( 1+d\right)
^{2/3}\right)=o(d^{-2}),$$
as $d\to\infty$.
Hence, our choice of $\varepsilon$ and $\theta$ satisfies that, for $d$ sufficiently large,
\begin{align*}
\phi _{d}\left( \theta \right) \exp \left( \chi \left( \theta \right)
+\varepsilon \log \left( d\right) \right)&\leq \left(1+\frac{\beta^2}{2d^2}\right)\left(1+\frac{\beta^2}{4d^2}\right)+o(d^{-2})\\
&\leq 1+\frac{\beta^2}{d^2}\leq \frac{1}{\left( 1-\beta /d\right) \left( 1+\beta /d\right) },
\end{align*}
which is exactly the inequality (\ref{Select_Theta}). On the other hand, note that $\epsilon\leq 1/2$, so when $d\geq 3$, we have
$$\chi \left( \theta \right) =\theta \frac{\delta _{1}}{2(1+d)}=\frac{1}{2\varepsilon+1}\cdot\frac{\delta^2_1\beta^2}{4d^2(1+d)^2\log(d)\max_{i=1}^d\sigma_i^2}\geq \frac{\delta^2_1\beta^2}{16d^4\log(d)\max_{i=1}^d\sigma_i^2}.$$
Now, let
\begin{equation*}
\zeta _{0}=\frac{\delta _{1}\cdot \beta^2 }{2\max_{i=1}^{d}\sigma _{i}^{2}},%
\text{ \ }\zeta _{1}=\frac{\delta _{1}^{2}\cdot \beta^2 }{16\max_{i=1}^{d}\sigma
_{i}^{2}}.
\end{equation*}%
According to Lemma \ref{Lem_LB_Nt_Bound} and the fact that $(1-p)^d\geq(1-\beta/d)^d\geq 1-\beta$, we have
\begin{align*}
E\left( 1-\beta \right) ^{\mathcal{N}\left( t;\mathbf{y}\right) }
&\leq \exp(h(\mathbf{y};\theta,\varepsilon)\exp(-\chi(\theta)t)(1-p)^{-d}\\
&\leq \exp(\theta\|\mathbf{y}\|_\infty+\varepsilon\log(d))\exp(-\chi(\theta)t)(1-p)^{-d}\\
&\leq  \exp
\left( \zeta _{0}\left\Vert \mathbf{y}\right\Vert _{\infty }/(d^{3}\log
\left( d\right) )+\beta /d^{2}\right) \cdot \exp \left( -\zeta
_{1}t/(d^{4}\log \left( d\right) )\right) \cdot \left( 1-\beta \right) ^{-1},
\end{align*}
where the second inequality follows Lemma \ref{Lem_LB_CT} and the last inequality follows our choice of $\theta$ and $\varepsilon$.
\end{proof}

\bigskip

We close this section with the proof of the technical results behind the proof
of Proposition \ref{prop:A2}.

\subsection{Technical Proofs of Auxiliary Results Behind Proposition
\ref{prop:A2}\label{Sec:proofs}}

We provide the proofs in the order in which we presented the auxiliary
results. First, the main ingredient behind Lemma \ref{Lem_LB_CT} is the
following result:\bigskip

\begin{lemma}
\label{Lemma_Aux_Mom}Suppose that there exists a non-negative and twice
continuously differentiable function $h\left(  \cdot\right)  $ and a constant
$\chi>0$ satisfying the following two conditions:\newline

\begin{enumerate}
\item For all $\mathbf{y}=(y_{1},...,y_{d})^{T}\in R_{+}^{d}$ with $\left\Vert
\mathbf{y}\right\Vert _{\infty}\geq1$%
\begin{equation}
\left(  \boldsymbol{\mu}-\boldsymbol{\mu}^{+}\right)  ^{T}Dh\left(
\mathbf{y}\right)  +\frac{1}{2}Tr\left(  \Sigma D^{2}h\left(  \mathbf{y}%
\right)  \right)  +\frac{1}{2}Dh\left(  \mathbf{y}\right)  ^{T}\Sigma
Dh\left(  \mathbf{y}\right)  \leq-\chi, \label{Sub_Sol}%
\end{equation}
where $Dh\left(  \mathbf{y}\right)  $ and $D^{2}h\left(  \mathbf{y}\right)  $
are the first and second derivatives of $h\left(  \cdot\right)  $ evaluated at
$\mathbf{y}$, respectively. (We encode $Dh\left(  \mathbf{y}\right)  $ as
column vector.) \newline

\item For any $\mathbf{y}=(y_{1},...,y_{d})^{T}\in\partial R_{+}^{d}$,
\begin{equation}
Dh\left(  \mathbf{y}\right)  ^{T}\mathbf{w}\leq0 \text{ for all }\mathbf{w}%
\in\mathcal{Z}_{\mathbf{y}}, \label{Ref_Cond}%
\end{equation}
where
\[
\mathcal{Z}_{\mathbf{y}}=\{\mathbf{w}=(w_{1},...,w_{d})^{T}\in R_{+}^{d}%
:w_{l}>0\text{ if and only if }y_{l}=0\}.
\]

\end{enumerate}

Then, for any $\left\Vert \mathbf{y}\right\Vert _{\infty}\geq1$,
\[
E\exp\left(  h\left(  \mathbf{Y}^{+}\left(  \tau^{+}\left(  \mathbf{y}\right)
\right)  \right)  +\chi\tau^{+}\left(  \mathbf{y}\right)  \right)  \leq
\exp\left(  h\left(  \mathbf{y}\right)  \right)  .
\]
In particular,
\[
E\exp\left(  \chi\tau^{+}\left(  \mathbf{y}\right)  \right)  \leq\exp\left(
h\left(  \mathbf{y}\right)  \right)  .
\]

\end{lemma}

\begin{proof}[Proof of Lemma \protect\ref{Lemma_Aux_Mom}]
Note that Ito's lemma yields that for a twice continuously differentiable $%
h\left( \cdot \right) $%
\begin{eqnarray}
\label{Dyn_1}&&~~~~~h\left( \mathbf{Y}^{+}\left( t\right) \right) -h\left( \mathbf{Y}%
^{+}\left( 0\right) \right)    \\
&=&\int_{0}^{t}\left( \mathcal{A}h\right) \left( \mathbf{Y}^{+}\left(
s\right) \right) ds+\int_{0}^{t}Dh\left( \mathbf{Y}^{+}\left( s\right)
\right) d\mathbf{L}^{+}\left( s\right) +\int_{0}^{t}Dh\left( \mathbf{Y}^+%
\left( s\right) \right) Cd\mathbf{B}\left( s\right) ,  \notag
\end{eqnarray}%
where $C$ is the Cholesky decomposition matrix such that $CC^{T}=\Sigma $, and%
\begin{equation*}
\left( \mathcal{A}h\right) \left( \mathbf{y}\right) ds\doteq\left( \boldsymbol{%
\mu }-\boldsymbol{\mu }^{+}\right) ^{T}Dh\left( \mathbf{y}\right) +\frac{1}{2%
}Tr\left( \Sigma D^{2}h\left( \mathbf{y}\right) \right) .
\end{equation*}%
We know that
\begin{equation*}
\bar{M}\left( t\right) =\exp \left( \int_{0}^{t}Dh\left( \mathbf{Y}^+\left(
s\right) \right) Cd\mathbf{B}\left( s\right) -\frac{1}{2}\int_{0}^{t}Dh%
\left( \mathbf{Y}^+\left( s\right) \right) ^{T}\Sigma Dh\left( \mathbf{Y}^+%
\left( s\right) \right) ds\right)
\end{equation*}%
is a non-negative local martingale and, therefore, a supermartingale. We
thus conclude that
\begin{equation*}
E_{\mathbf{y}}\bar{M}\left( t\right) \leq 1.
\end{equation*}%
Since $dL^+_i(s)\geq 0$ and $Y^+_i(s)dL^+_i(s)=0$ for all $1\leq i\leq d$ and $s\geq 0$, $dL^+_i(s)=0$ for all $Y^+_i(s)>0$. Besides, $P(\exists \delta>0, \text{ such that }X_i(s+t)-X_i(s)\geq 0\text{ for all }0\leq t\leq \delta)=0$ by the nature of Brownian motion, and hence $dL^+_i(s)>0$ w.p.1 for $Y^+_i(s)=0$. Therefore, we  know that $d\mathbf{L}^+(s)\in \mathcal{Z}_{\mathbf{Y}^+(s)}$ for all $s\geq 0$. Under Condition 2, $\int Dh(\mathbf{Y}^+(s))d\mathbf{L}^+(s)\leq 0$.
Substituting (\ref{Dyn_1}) into $\bar{M}\left( t\right) $, we obtain that%
\begin{equation*}
E_{\mathbf{y}}\exp \left( h\left( \mathbf{Y}^{+}\left( \tau ^{+}(\mathbf{y}%
)\right) \right) -h\left( \mathbf{y}\right) +\chi \tau ^{+}(\mathbf{y}%
)\right) \leq E_{\mathbf{y}}\bar{M}\left( t\right) \leq 1.
\end{equation*}%
Because $h\left( \cdot \right) \geq 0$, we conclude that
\begin{equation*}
E_{\mathbf{y}}\exp \left( -h\left( \mathbf{y}\right) +\chi \tau ^{+}(\mathbf{%
y})\right) \leq 1,
\end{equation*}%
which is equivalent to the statement of the result.
\end{proof}

\bigskip

Using the previous result, we now can provide the proof of Lemma
\ref{Lem_LB_CT}.

\begin{proof}[Proof of Lemma \protect\ref{Lem_LB_CT}]
We start by computing the first and second derivatives of $h\left( \cdot
\right) $. Let%
\begin{equation*}
\mathbf{\ }w_{i}(y,\varepsilon )=\frac{\exp (g(\theta y_{i})/\varepsilon )}{%
\sum_{j=1}^{d}\exp (g(\theta y_{j})/\varepsilon )}.
\end{equation*}%
Note that
\begin{eqnarray*}
Dh\left( \mathbf{y}\right)  &=&\sum_{i=1}^{d}w_{i}\left( y,\varepsilon
\right) g^{\prime }\left( \theta e_{i}^{T}\mathbf{y}\right) \theta e_{i} \\
D^{2}h\left( \mathbf{y}\right)  &=&\theta ^{2}\sum_{i=1}^{d}w_{i}\left(
y,\varepsilon \right) g^{\prime \prime }\left( \theta e_{i}^{T}\mathbf{y}%
\right) e_{i}e_{i}^{T} \\
&&+\frac{\theta ^{2}}{\varepsilon }\sum_{i=1}^{d}w_{i}\left( y,\varepsilon
\right) g^{\prime }\left( \theta e_{i}^{T}\mathbf{y}\right)
^{2}e_{i}e_{i}^{T} \\
&&-\frac{\theta ^{2}}{\varepsilon }\sum_{i,j=1}^{d}w_{i}\left( y,\varepsilon
\right) w_{j}\left( y,\varepsilon \right) g^{\prime }\left( \theta e_{i}^{T}%
\mathbf{y}\right) g^{\prime }\left( \theta e_{j}^{T}\mathbf{y}\right)
e_{i}e_{j}^{T}.
\end{eqnarray*}%
Therefore,
\begin{eqnarray*}
Tr\left( \Sigma D^{2}h\left( \mathbf{y}\right) \right)  &\leq &\frac{\theta
^{2}}{\varepsilon }\sum_{i=1}^{d}w_{i}\left( y,\varepsilon \right) \sigma
_{i}^{2}(\varepsilon g^{\prime \prime }\left( \theta e_{i}^{T}\mathbf{y}%
\right) +g^{\prime }\left( \theta e_{i}^{T}\mathbf{y}\right) ^{2}), \\
Dh\left( \mathbf{y}\right) ^{T}\Sigma Dh\left( \mathbf{y}\right)  &=&\theta
^{2}\left(\sum_{i=1}^{d}w_{i}\left( y,\varepsilon \right) g^{\prime }\left( \theta
e_{i}^{T}\mathbf{y}\right) \sigma _{i}\right)^{2}\leq \theta
^{2}\max_{i=1}^{d}\sigma _{i}^{2}, \\
\left( \boldsymbol{\mu }-\boldsymbol{\mu }^{+}\right) ^{T}Dh\left( \mathbf{y}%
\right)  &=&-\theta \sum_{i=1}^{d}w_{i}\left( y,\varepsilon \right)
g^{\prime }\left( \theta e_{i}^{T}\mathbf{y}\right) \delta _{1}.
\end{eqnarray*}%
Because $-w_{i}\left( y,\varepsilon \right) g^{\prime }\left( \theta
e_{i}^{T}\mathbf{y}\right) \delta _{1}\leq 0$, we have that
\begin{eqnarray*}
-\sum_{i=1}^{d}w_{i}\left( y,\varepsilon \right) g^{\prime }\left( \theta
e_{i}^{T}\mathbf{y}\right) \delta _{1} &\leq &-\sum_{i=1}^{d}w_{i}\left(
y,\varepsilon \right) g^{\prime }\left( \theta e_{i}^{T}\mathbf{y}\right)
\delta _{1}I\left( e_{i}^{T}\mathbf{y}\geq 1\right)  \\
&=&-\sum_{i=1}^{d}w_{i}\left( y,\varepsilon \right) \delta _{1}I\left(
e_{i}^{T}\mathbf{y}\geq 1\right)  \\
&\leq &\frac{-\delta _{1}}{1+d},
\end{eqnarray*}%
where in the last inequality we use the fact that, for $\left\Vert \mathbf{y}%
\right\Vert _{\infty }\geq 1$,
\begin{equation*}
\sum_{i=1}^{d}w_{i}\left( y,\varepsilon \right) I\left( e_{i}^{T}\mathbf{y}%
\geq 1\right) \geq \frac{1}{d+1}.
\end{equation*}%
On the other hand,
\begin{eqnarray*}
\sum_{i=1}^{d}w_{i}\left( y,\varepsilon \right) \sigma _{i}^{2}g^{\prime
\prime }\left( \theta e_{i}^{T}\mathbf{y}\right)  &\leq
&\sum_{i=1}^{d}w_{i}\left( y,\varepsilon \right) \sigma _{i}^{2}\leq
\max_{i=1}^{d}\sigma _{i}^{2}, \\
\sum_{i=1}^{d}w_{i}\left( y,\varepsilon \right) \sigma _{i}^{2}g^{\prime
}\left( \theta e_{i}^{T}\mathbf{y}\right) ^{2} &\leq
&\sum_{i=1}^{d}w_{i}\left( y,\varepsilon \right) \sigma _{i}^{2}\leq
\max_{i=1}^{d}\sigma _{i}^{2}.
\end{eqnarray*}%
We conclude that%
\begin{eqnarray*}
&&\left( \boldsymbol{\mu }-\boldsymbol{\mu }^{+}\right) ^{T}Dh\left( \mathbf{%
y}\right) +\frac{1}{2}Tr\left( \Sigma D^{2}h\left( \mathbf{y}\right) \right)
+\frac{1}{2}Dh\left( \mathbf{y}\right) ^{T}\Sigma Dh\left( \mathbf{y}\right)
\\
&\leq &\frac{-\theta }{1+d}\delta _{1}+\theta ^{2}\frac{1}{2}%
\max_{i=1}^{d}\sigma _{i}^{2}+\theta ^{2}\frac{1}{2\varepsilon }%
\max_{i=1}^{d}\sigma _{i}^{2}+\frac{\theta ^{2}}{2}\max_{i=1}^{d}\sigma
_{i}^{2} \\
&\leq &\theta \max_{i=1}^{d}\sigma _{i}^{2}\cdot \left( \theta \left( \frac{%
2\varepsilon +1}{2\varepsilon }\right) -\frac{\delta _{1}}{%
\max_{i=1}^{d}\sigma _{i}^{2}(1+d)}\right) \leq -\theta \frac{\delta _{1}}{%
2(1+d)},
\end{eqnarray*}%
assuming that%
\begin{equation*}
\theta \leq \frac{\varepsilon }{2\varepsilon +1}\cdot \frac{\delta _{1}}{%
(1+d)\max_{i=1}^{d}\sigma _{i}^{2}}.
\end{equation*}%
Therefore, we conclude that the condition (\ref{Sub_Sol})\ holds for $\left\Vert
\mathbf{y}\right\Vert _{\infty }\geq 1$. On the other hand, since $Dh\left( \mathbf{y}%
\right) ^{T}e_{i}=g^{\prime }\left( \theta y_{i}\right) =0$, if $y_{i}=0$, we
also satisfy (\ref{Ref_Cond}). Finally, we apply Lemma 9 and conclude  \eqref{in_0_Mg}.
\end{proof}

\bigskip Now, we prove the success probability of coupon collection is
uniformly bounded from 0.

\begin{proof}[Proof of Lemma \protect\ref{Lem:Coupon}]
For any fixed $i\in \{1,...,d\}$, note that the event $Y_{i}(t)=0\text{ for some
}t\leq 1$ is equivalent to $L_{i}(1)>0$ and hence
\begin{equation*}
P(Y_{i}(t)=0\text{ for some }t\leq 1|\mathbf{Y}(0)=\mathbf{y})=P(L_{i}(1)>0|%
\mathbf{Y}(0)=\mathbf{y}).
\end{equation*}
Let $\mathbf{Z}(t)=R^{-1}(y_{0}+\mathbf{X}(t))$. Define $(\mathbf{Y}^*,\mathbf{L}^*)$ to be the solution to the following Skorokhod problem:
\begin{equation*}
\mathbf{Y}^{\ast }(t)=\mathbf{Z}(t)+\mathbf{L}^{\ast }(t)\geq 0, %
\mathbf{L}^{\ast }(t)= 0.
\end{equation*}%
In particular, the process $\mathbf{L}^*(\cdot)$ is nondecreasing and $Y_{i}^{\ast }(t)dL_{i}^{\ast }(t)=0$ for all $t\geq 0$.
Then,  $\mathbf{L}^{\ast }(t)$ is the minimal process that keeps $%
\mathbf{Y}^{\ast }(t)$ non-negative. Note that $R^{-1}\mathbf{Y}(t)=\mathbf{Z}%
(t)+\mathbf{L}(t)\geq 0$, therefore
\begin{equation*}
L_{i}(t)\geq L_{i}^{\ast }(t)\text{ and }Y_{i}(t)\geq Y_{i}^{\ast }(t).
\end{equation*}%
As a result,
\begin{equation*}
P(Y_{i}(t)=0\text{ for some }t\leq 1|\mathbf{Y}(0)=\mathbf{y})=P(L_{i}(1)>0|%
\mathbf{Y}(0)=\mathbf{y})\geq P(L_{i}^{\ast }(1)>0|\mathbf{Y}(0)=%
\mathbf{y}).
\end{equation*}%
By definition,
\begin{equation*}
P(L_{i}^{\ast }(1)>0|\mathbf{Y}(0)=\mathbf{y})\geq P(Z_{i}(1)<0)=P((R^{-1}%
\mathbf{y}+R^{-1}\boldsymbol{\mu }+R^{-1}C\mathbf{B}(1))_{i}<0).
\end{equation*}%
Note that following Assumption A1), $\Vert R^{-1}\mathbf{1}\Vert _{\infty
}\leq b_{1}$ and $\mathbf{y}\leq b_{1}\mathbf{1}$. Therefore, $R^{-1}%
\mathbf{y}\leq b_{1}^{2}\mathbf{1}$. Since $R^{-1}\boldsymbol{\mu }\leq
-\delta _{0}$, we have $(R^{-1}\mathbf{y}+R^{-1}\boldsymbol{\mu })_{i}\leq
b_{1}^{2}-\delta _{0}$ and hence
\begin{equation*}
P((R^{-1}\mathbf{y}+R^{-1}\boldsymbol{\mu }+R^{-1}C\mathbf{B}(1))_{i}<0)\geq
P((R^{-1}C\mathbf{B}(1))_{i}<\delta _{0}-b_{1}^{2}).
\end{equation*}%
Since $R_{ii}^{-1}\geq 1$ and $\sigma _{i}^{2}\geq b_{0}^{-1}$ according to
Assumption A3), $(R^{-1}C\mathbf{B}(1))_{i}$ is a Gaussian r.v. with
variance $\geq b_{0}^{-1}$. Therefore, we conclude that
\begin{eqnarray*}
P(Y_{i}(t) &=&0\text{ for some }t\leq 1|\mathbf{Y}(0)=\mathbf{y}) \\
&\geq &P((R^{-1}C\mathbf{B}(1))_{i}<\delta _{0}-b_{1}^{2}) \\
&\geq &P\left( N\left( 0,1\right) <\sqrt{b_{0}}(\delta
_{0}-b_{1}^{2})\right) \geq p_{0}.
\end{eqnarray*}
\end{proof}

\bigskip We continue with the proof of Lemma \ref{Lemma_Stoch_Dom}.

\begin{proof}[Proof of Lemma \protect\ref{Lemma_Stoch_Dom}]
Recall that we have defined a sequence of Bernoulli random variables $\zeta_k(i)$ jointly with the sequence $\{\mathbf{Y}^{i,k}\}$ as
$$\zeta_k(i)= 1(Y_i(t;\mathbf{Y}^{i,k})=0\text{ for some }\tau^+(k;\mathbf{Y}^{i,k})< t\leq \tau^+(k;\mathbf{Y}^{i,k})+1).$$
Let $K=\min\{k: \zeta_k(i)=1\}$. We obtain a bound for $\eta^1_i(\mathbf{y})$:
$\eta^1_i(\mathbf{y})\leq \sum_{k=1}^K (\tau^+(k;\mathbf{Y}^{i,k})+1)$.
Note that the the Skorokhod mapping is monotone with respect to the initial position, i.e.,
$$\mathbf{Y}(t;\mathbf{y}^1)\geq \mathbf{Y}(t;\mathbf{y}^2) \text{ for all }t\geq 0\text{ if }\mathbf{y}^1\geq \mathbf{y}^2.$$
As $\mathbf{Y}(\tau^+(k;\mathbf{Y}^{i,k});\mathbf{Y}^{i,k})\leq b_1\mathbf{1}$, we have
$$\mathbf{Y}^{i,k+1}\leq \mathbf{Y}(1;b_1\mathbf{1})\text{, and hence }\tau^+(k+1;\mathbf{Y}^{i,k+1})\leq \tau^+(k+1;\mathbf{Y}(1;b_1\mathbf{1})).$$
Similarly, we have
$$P(\zeta_k(i)=1)\geq P(Y_i(t;b_1\mathbf{1})=0\text{ for some }0< t\leq 1)\geq p,$$
where the last inequality follows Lemma \ref{Lem:Coupon}. As a result, we can define a Bernoulli $\psi$ jointly with $\mathbf{Y}(1;b_1\mathbf{1})$, such that for all $\mathbf{y}\geq 0$
$$P(\psi=1|\mathbf{Y}(1;b_1\mathbf{1})=\mathbf{y})\leq P(Y_i(t;b_1\mathbf{1})=0\text{ for some }0< t\leq 1 |\mathbf{Y}(1;b_1\mathbf{1})=\mathbf{y}),$$
and $P(\psi=1)=p.$
Based on the previous comparison results, we can construct a sequence of pairs $(\psi_k(i), \tau_k(i))$ to be i.i.d. copies of $(\psi, \tau^+(\mathbf{Y}(1;b_1\mathbf{1})))$, for $1\leq j\leq d$ and $k\geq 1$, and define $G^i=\inf\{k:\psi_k(i)=1\}$. Then $G^i$ is a Geometric r.v. with probability of success equal to $p$, and  $\eta^1(\mathbf{y})$ is stochastically dominated by
$$\tau^+(\mathbf{y})+\sum_{i=1}^d\sum_{k=1}^{G^i}(1+\tau_k(i)).$$
Since $(\psi_k(i), \tau_k(i))$ are i.i.d., we have that
$$\sum_{i=1}^d\sum_{k=1}^{G^i}\tau_k(i)\overset{D}{=}\sum_{i=1}^d\sum_{k=1}^{G^i}\tilde{\tau}_k(i),$$
where  for each $i$, $\{\tilde{\tau}_k(i):k\geq 1\}$ is an i.i.d. sequence following the conditional distribution of $\tau_k(i)$ conditional on that $\psi_k(i)=0$ and is independent of $G^i$.
The rest of the proof is to construct the r.v. $\Theta_d$ satisfying \eqref{Bnd_fhi_d} and that $\tau^+(\Theta_d\mathbf{1})$ stochastically dominates $\tilde{\tau}_k(i)$.
Recall that $\mathbf{Y}(1)\leq R^{-1}\mathbf{Y}(1)\leq R^{-1}\mathbf{Y}^{+}(1)$ and $%
\bar{\mathbf{X}}(t)\leq C\mathbf{B}(t)%
$ for all $t>0$, where $\mathbf{B}(t)$ is a standard Brownian motion. By the property of the Skorokhod mapping
with the identity reflection matrix, we have
\begin{align*}
Y_{i}^{+}(1)&= Y_{i}^{+}(0)+X_i(1)-\left(\inf_{0\leq t\leq 1}(Y_{i}^{+}(0)+X_i(t))\right)\wedge 0\\
&\leq Y_{i}^{+}(0)+X_i(1)-\left(\inf_{0\leq t\leq 1} X_i(1)\right)\wedge 0\\
&= Y_{i}^{+}(0)+e_{i}^{T}C\mathbf{B}(1)-\inf_{0\leq t\leq 1}e_{i}^{T}C\mathbf{B}(t).\end{align*}%
Let us write $\mathbf{U}=C\mathbf{B}(1)-\inf_{0\leq t\leq 1}C\mathbf{B}(t)$, so whenever $\mathbf{Y}(0)=\mathbf{y}\leq b_1\mathbf{1}$, we have
\begin{equation*}
\mathbf{Y}(1;\mathbf{y})\leq _{st}(b_{1}+b_{1}\left\Vert \mathbf{U}%
\right\Vert _{\infty })\mathbf{1}.
\end{equation*}%
Now we define $\Theta_d>0$ as
$$P(\Theta_d>t)=\min\left(1,\frac{P(b_1+b_1\|\mathbf{U}\|_\infty>t)}{1-p}\right)\text{ for all }t>0.$$
Recall that $\tau_k(i)$ is a copy of $\tau^+(\mathbf{Y}(1;b_1\mathbf{1}))$, and $\tau^+(\mathbf{y}^1)\geq_{st}\tau^+(\mathbf{y}^2)$ whenever $\mathbf{y}^1\geq\mathbf{y}^2$. Therefore,
$$P(\tau_k(i)> t)\leq P(\tau^+((b_1+b_1\|\mathbf{U}\|_\infty)\mathbf{1})>t)\leq (1-p)P(\tau^+(\Theta_d\mathbf{1})>t).$$
For all $t>0$,
$$P(\tilde{\tau}_{k}(i)> t)=P(\tau_k(i)> t|\psi_{k}(i)=0)\leq \frac{P(\tau_k(i)> t)}{1-p}\leq P(\tau^+(\Theta_d\mathbf{1})>t).$$
Now we show that $\Theta$ satisfies \eqref{Bnd_fhi_d}. Note that%
\begin{equation*}
E\exp \left( \theta \left\Vert \mathbf{U}\right\Vert _{\infty }\right)
=\int_{0}^{\infty }\theta \exp \left( \theta t\right) P\left( \left\Vert
\mathbf{U}\right\Vert _{\infty }>t\right) dt+1
\end{equation*}%
If $t=s\log \left( 1+d\right) $, breaking the integral on $[0,1/\log
(1+d)^{1/3}]$ and $(1/\log (1+d)^{1/3},\infty )$, we obtain
\begin{eqnarray*}
&&\int_{0}^{\infty }\theta \exp \left( \theta t\right) P\left( \left\Vert
\mathbf{U}\right\Vert _{\infty }>t\right) dt \\
&=&\theta \log \left( 1+d\right) \int_{0}^{\infty }\exp \left( s\theta \log
\left( 1+d\right) \right) P\left( \left\Vert \mathbf{U}\right\Vert _{\infty
}>s\log \left( 1+d\right) \right) ds \\
&\leq &\theta \log \left( 1+d\right) \exp \left( \theta \log \left(
1+d\right) ^{2/3}\right)  \\
&&+\theta \log \left( 1+d\right) \int_{1/\log \left( 1+d\right)
^{1/3}}^{\infty }\exp \left( s\theta \log
\left( 1+d\right) \right) P\left( \left\Vert \mathbf{U}\right\Vert _{\infty }>s\log
\left( 1+d\right) \right) ds.
\end{eqnarray*}%
Since $U_{i}=e_{i}^{T}C\mathbf{B}(1)-\inf_{0\leq t\leq 1}e_{i}^{T}C\mathbf{B}(t)=\sup_{0\leq t\leq 1}e_i^TC(\mathbf{B}(1)-\mathbf{B}(t))$ is equal in distribution to $\sup_{0\leq t\leq 1}e_{i}^{T}C\mathbf{B}(t) $, by the reflection principle for Brownian motions, we have
\begin{equation*}
P\left( U_{i}>t\right) =2\int_{t}^{\infty }\frac{1}{\sqrt{2\pi }\sigma _{i}}%
\exp (-r^{2}/2b_{1}\sigma _{i}^{2})dr\leq \frac{2\sigma _{i}}{t\sqrt{2\pi }}%
\exp (-t^{2}/2\sigma _{i}^{2})\leq \frac{2\sqrt{b_{0}}}{t\sqrt{2\pi }}\exp
(-t^{2}/2b_{0}).
\end{equation*}%
Therefore,
\begin{eqnarray*}
&&\int_{1/\log \left( 1+d\right) ^{1/3}}^{\infty }\exp \left( s\theta \log
\left( 1+d\right) \right) P\left( \left\Vert \mathbf{%
U}\right\Vert _{\infty }>s\log \left( 1+d\right) \right) ds \\
&\leq &d\int_{1/\log \left( 1+d\right) ^{1/3}}^{\infty }\frac{2\sqrt{b_{0}}}{%
s\log (1+d)\sqrt{2\pi }}\exp \left( -\frac{s^{2}\log \left( 1+d\right) ^{2}}{%
2b_{0}}+s\theta \log
\left( 1+d\right)\right) ds \\
&\leq &\frac{2\sqrt{b_{0}}d}{\log (1+d)^{2/3}\sqrt{2\pi }}\exp \left( -\log \left(
1+d\right) ^{4/3}/3b_{0}\right) ,
\end{eqnarray*}%
as $\theta=o(1)$ and hence $s\theta \log
\left( 1+d\right)\leq s^{2}\log \left( 1+d\right) ^{2}/6b_0$ for $d$ that is large enough.
%
%
%
%
Therefore, we conclude that
\begin{equation*}
\phi _{d}\left( \theta \right) \leq 1+2(1-p)^{-1}\theta \log \left(
1+d\right) \exp \left( \theta \log \left( 1+d\right) ^{2/3}\right) +\theta
O\left( d\exp \left( -\log \left( 1+d\right) ^{4/3}/3b_{0}\right)
\right) .
\end{equation*}
\end{proof}

\bigskip\begin{proof}[Proof of Lemma \ref{Lem_LB_DT}]
Observe that%
\begin{eqnarray*}
&&E\exp \left( \chi \left( \theta \right) \left( \tau ^{+}\left( \Lambda
_{d}^{k}\left( j\right) \mathbf{1}\right) +1\right) \right) \leq \exp
\left( \chi \left( \theta \right) \right) E\exp \left( h\left( \Lambda
_{d}^{k}\left( j\right) \mathbf{1,}\theta \right) \right) \\
&\leq& \exp \left( \chi \left( \theta \right) +\varepsilon \log d\right)
E\exp \left( \theta \Lambda _{d}\right) =\exp \left( \chi \left( \theta
\right) +\varepsilon \log d\right) \phi _{d}\left( \theta \right) .
\end{eqnarray*}%
Therefore,
\begin{equation}
E\exp \left( \chi \left( \theta \right) \xi \right) \leq \left( \frac{\phi
_{d}\left( \theta \right) \exp \left( \chi \left( \theta \right)
+\varepsilon \log \left( d\right) \right) p}{1-(1-p)\phi _{d}\left(
\theta \right) \exp \left( \chi \left( \theta \right) +\varepsilon \log
\left( d\right) \right) }\right) ^{d}.  \label{In_1_Mgale}
\end{equation}%
Since $\eta^n(\mathbf{y})\leq \tau^+(y)+\xi_1+...+\xi_n$ where $\tau^+(y)$, $\xi_1$,..., $\xi_n$ are all independent of each other, and $E\exp(\chi(\theta)\tau^+(\mathbf{y}))\leq \exp(h(\mathbf{y};\theta,\varepsilon))$ by Lemma \ref{Lem_LB_CT}, we have
\begin{equation*}
E\exp \left( \chi \left( \theta \right) \eta ^{n}\left( \mathbf{y}\right)
\right) \leq \exp \left( h\left( \mathbf{y};\theta,\varepsilon \right) \right) \left(
\frac{\phi _{d}\left( \theta \right) \exp \left( \chi(\theta) +\varepsilon \log
\left( d\right) \right) p}{1-(1-p)\phi _{d}\left( \theta \right)
\exp \left( \chi(\theta) +\varepsilon \log \left( d\right) \right) }\right) ^{nd}.
\end{equation*}
Since the function $f(x)\doteq xp/(1-(1-p)x)$ is increasing in $x$ for $x\leq 1/(1-p)$, under \eqref{Select_Theta}, we have
$$\frac{[\phi _{d}\left( \theta \right) \exp \left( \chi(\theta) +\varepsilon \log
\left( d\right) \right) ]p}{1-(1-p)[\phi _{d}\left( \theta \right)
\exp \left( \chi(\theta) +\varepsilon \log \left( d\right) \right)] }\leq f\left(\frac{1}{(1-p)(1+p)}\right)=\frac{1}{1-p},$$
and we are done.
\end{proof}

We conclude this section with the proof of Lemma \ref{Lem_LB_Nt_Bound}.

\begin{proof}[Proof of Lemma \protect\ref{Lem_LB_Nt_Bound}]
Let us write $\xi _{0}=\tau^{+}\left( \mathbf{y}\right) $, $A_{-1}=0$,
and $A_{n}=\xi _{0}+...+\xi _{n}$
\begin{equation*}
\bar{N}\left( t\right) =\sup \{n\geq -1:A_{n}\leq t\},
\end{equation*}%
so that $\bar{N}\left( \cdot \right) $ is a delayed renewal process.
Following Lemma 6 and \eqref{Stoch_Dom}, we have $\eta^n(y)\leq_{st} \tau^+(y) +\xi_1 +...+\xi_n=A_n$. On the other hand, for $A_n\leq t<A_{n+1}$, by defintion, $\bar{N}(t)=n$ and $\mathcal{N}(t;\mathbf{y})\geq\mathcal{N}(A_n;\mathbf{y})\geq_{st} \mathcal{N}(\eta^n(\mathbf{y});\mathbf{y}) = n$.
So we have $\bar{N}\left( t\right) \leq _{st}\mathcal{N}\left( t;%
\mathbf{y}\right) $ and, therefore, for any $\beta >0$,%
\begin{equation*}
E\exp \left( -\beta \mathcal{N}\left( t;\mathbf{y}\right) \right) \leq E\exp
\left( -\beta \bar{N}\left( t\right) \right) .
\end{equation*}%
According to Lemma 6 and Lemma 7,
\begin{equation*}
M_{n}=\exp \left( \chi \left( \theta \right) A_{n}-h\left( \mathbf{y;}\theta,\varepsilon
\right) \right) \left( 1-p\right) ^{dn}
\end{equation*}%
is a non-negative supermartingale and, therefore,
\begin{equation*}
1\geq EM_{\bar{N}\left( t\right) +1}\geq E\left( \exp \left( \chi \left(
\theta \right) t-h\left( \mathbf{y;}\theta,\varepsilon\right) \right) \left(
1-p\right) ^{\left( \bar{N}\left( t\right) +1\right) d}\right) ,
\end{equation*}%
thereby concluding that
\begin{equation*}
E\left( \left( 1-p\right) ^{d\cdot \mathcal{N}\left( t;\mathbf{y}\right) }\right)\leq E\left( \left( 1-p\right) ^{d\cdot \bar{N}\left( t\right) }\right) \leq
\exp \left( h\left( \mathbf{y;}\theta ,\varepsilon\right) \right) \cdot \exp \left(
-\chi \left( \theta \right) t\right) \cdot \left( 1-p\right) ^{-d},
\end{equation*}%
and the result follows.
\end{proof}

\bigskip

\section{Step 3: Concluding the Proof of Theorem \ref{Thm_Main}%
\label{Section_Step_3}}

For any $f\in\mathcal{L}$,
\begin{align*}
&  E\left\vert f\left(  \mathbf{Y}\left(  t;\mathbf{y}\right)  \right)
-f\left(  \mathbf{Y}\left(  t;\mathbf{Y}\left(  \infty\right)  \right)
\right)  \right\vert \\
&  \leq E\left\vert \left\vert \mathbf{Y}\left(  t;\mathbf{y}\right)
-\mathbf{Y}\left(  t;\mathbf{Y}\left(  \infty\right)  \right)  \right\vert
\right\vert _{\infty}\\
&  \leq E\left\vert \left\vert \mathbf{Y}\left(  t;\mathbf{y}\right)
-\mathbf{Y}\left(  t;\mathbf{0}\right)  \right\vert \right\vert _{1}%
+E\left\vert \left\vert \mathbf{Y}\left(  t;\mathbf{0}\right)  -\mathbf{Y}%
\left(  t;\mathbf{Y}\left(  \infty\right)  \right)  \right\vert \right\vert
_{1}.
\end{align*}
Therefore, by Lemma \ref{Lem_DB1}, we have that
\begin{align}
&  ~~E\left\vert f\left(  \mathbf{Y}\left(  t;\mathbf{y}\right)  \right)
-f\left(  \mathbf{Y}\left(  t;\mathbf{Y}\left(  \infty\right)  \right)
\right)  \right\vert \label{B}\\
\leq &  ~d\cdot\kappa_{0}\cdot\left(  E\left(  \left(  1-\beta_{0}\right)
^{\mathcal{N}\left(  t;\mathbf{y}\right)  }\left\Vert \mathbf{y}\right\Vert
_{1}\right)  +E\left(  \left(  1-\beta_{0}\right)  ^{\mathcal{N}\left(
t;\mathbf{Y}\left(  \infty\right)  \right)  }\left\Vert \mathbf{Y}\left(
\infty\right)  \right\Vert _{1}\right)  \right)  .\nonumber
\end{align}
For the last term, according to the Cauchy-Schwarz inequality, we have that
\[
E\left(  \left(  1-\beta_{0}\right)  ^{\mathcal{N}\left(  t;\mathbf{Y}\left(
\infty\right)  \right)  }\left\Vert \mathbf{Y}\left(  \infty\right)
\right\Vert _{1}\right)  \leq E^{1/2}\left(  \left\Vert \mathbf{Y}\left(
\infty\right)  \right\Vert _{1}^{2}\right)  E^{1/2}\left(  \left(  1-\beta
_{0}\right)  ^{2\mathcal{N}\left(  t;\mathbf{Y}\left(  \infty\right)  \right)
}\right)  .
\]
Following the stochastic domination result (\ref{SD_KW_96}) and the fact that
$R^{-1}\geq I$, we have
\[
\left\Vert \mathbf{Y}\left(  \infty\right)  \right\Vert _{1}\leq\left\Vert
R^{-1}\mathbf{Y}\left(  \infty\right)  \right\Vert _{1}\leq\left\Vert
R^{-1}\mathbf{Y}^{+}\left(  \infty\right)  \right\Vert _{1}\leq\left\Vert
R^{-1}\mathbf{1}\right\Vert _{1}\left\Vert \mathbf{Y}^{+}\left(  \infty\right)
\right\Vert _{1}\leq db_{1}\left\Vert \mathbf{Y}^{+}\left(  \infty\right)
\right\Vert _{1}.
\]
Moreover,
\[
\left\Vert \mathbf{Y}^{+}\left(  \infty\right)  \right\Vert _{1}^{2}=\left(
\sum_{i=1}^{d}Y_{i}^{+}\left(  \infty\right)  \right)  ^{2}\leq d\left(
\sum_{i=1}^{d}Y_{i}^{+}\left(  \infty\right)  ^{2}\right) .
\]
By definition, $Y_{i}^{+}\left(  \infty\right)  $ represents a one-dimensional
RBM, with drift $-(\mu_{i}^{+}-\mu_{i})$ and variance $\sigma_{i}^{2}$, in its
steady state. So $Y_{i}^{+}(\infty)$ follows an exponential distribution with
mean $\sigma_{i}^{2}/2\left(  \mu_{i}^{+}-\mu_{i}\right)  $, and therefore
(recall that we have chosen $\delta_{1}=\delta_{0}\beta_{0}/2\kappa_{0}$),%
\[
E\left(  Y_{i}^{+}\left(  \infty\right)^2\right)  =\frac{\sigma_{i}^{4}%
}{\left(  \mu_{i}^{+}-\mu_{i}\right) ^{2} }=\frac{\sigma_{i}^{4}}{\delta_{1}^{2}}%
\leq\left(\frac{2\sigma_{i}^{2}\kappa_{0}}{\delta_{0}\beta_{0}}\right)^{2}\leq
\left(\frac
{2b_{0}\kappa_{0}}{\delta_{0}\beta_{0}}\right)^{2},
\]
which concludes that
\begin{equation}
E^{1/2}\left(  \left\Vert \mathbf{Y}\left(  \infty\right)  \right\Vert
_{1}^{2}\right)  \leq 2\cdot d^2\cdot\frac{\kappa_{0}}%
{\delta_{0}\beta_{0}}b_{0}b_1 . \label{B_Aux_LL}%
\end{equation}
Next, invoking Proposition \ref{prop:A2} with $\beta\in\left(  0,\min\left(
\beta_{0},1/3\right)  \cdot1/3\right)  $, we can  guarantee that $\left(
1-\beta\right)  \geq\left(  1-\beta_{0}\right)  ^{2}$, and therefore conclude
that%
\begin{align}
E\left(  \left(  1-\beta_{0}\right)  ^{2\mathcal{N}\left(  t;\mathbf{Y}\left(
\infty\right)  \right)  }\right)   &  \leq E\left[  \exp\left(  \zeta
_{0}\left\Vert \mathbf{Y}\left(  \infty\right)  \right\Vert _{\infty}%
/(d^{3}\log\left(  d\right)  )+\beta/d^{2}\right)  \right]  \times
\label{BBB}\\
&  ~\exp\left(  -\zeta_{1}t/(d^{4}\log\left(  d\right)  )\right)  \cdot\left(
1-\beta\right)  ^{-1},\nonumber
\end{align}
where%
\[
\zeta_{0}=\frac{\delta_{1}\cdot\beta}{2\max_{i=1}^{d}\sigma_{i}^{2}},\text{
\ }\zeta_{1}=\frac{\delta_{1}^{2}\cdot\beta}{16\max_{i=1}^{d}\sigma_{i}^{2}}.
\]
Once again, using the stochastic domination result (\ref{SD_KW_96}), we have
that%
\[
\left\Vert \mathbf{Y}\left(  \infty\right)  \right\Vert _{\infty}%
\leq\left\Vert R^{-1}\mathbf{Y}\left(  \infty\right)  \right\Vert _{\infty
}\leq\left\Vert R^{-1}\mathbf{Y}^{+}\left(  \infty\right)  \right\Vert
_{\infty}\leq \|\mathbf{1}R^{-1}\|_\infty\|\mathbf{Y}^+(\infty)\|_\infty\leq d b_{1}\left\Vert \mathbf{Y}^{+}\left(  \infty\right)  \right\Vert
_{\infty}.
\]
Observe that
\[
P\left(  \left\Vert \mathbf{Y}^{+}\left(  \infty\right)  \right\Vert _{\infty
}>t\right)  \leq\sum_{i=1}^{d}P\left(  Y_{i}^{+}\left(  \infty\right)
>t\right)  \leq d\exp\left(  -\frac{2\delta_{1}}{\max_{i=1}^{d}\sigma_{i}^{2}%
}t\right)  .
\]
We conclude that
\begin{align*}
&  E[\exp\left(  \zeta_{0}\left\Vert \mathbf{Y}\left(  \infty\right)
\right\Vert _{\infty}/(d^{3}\log\left(  d\right)  )\right)  ]\\
&  \leq\frac{\zeta_{0}}{d^{3}\log\left(  d\right)  }\int_{0}^{\infty}%
\exp\left(  \frac{\zeta_{0}}{(d^{3}\log\left(  d\right)  }t\right)  P\left(
\left\Vert \mathbf{Y}^{+}\left(  \infty\right)  \right\Vert _{\infty
}>t\right)  dt+1\\
&  \leq\frac{\zeta_{0}d}{d^{3}\log\left(  d\right)  }\int_{0}^{\infty}%
\exp\left(  \frac{\zeta_{0}}{(d^{3}\log\left(  d\right)  }t-\frac{2\delta_{1}%
}{\max_{i=1}^{d}\sigma_{i}^{2}db_{1}}t\right)  dt+1\\
& = \frac{\zeta_{0}d^{2}}{d^{3}\log\left(  d\right)  }\int_{0}^{\infty}%
\exp\left(  \frac{\zeta_{0}}{(d^{2}\log\left(  d\right)  }t-\frac{2\delta_{1}%
}{\max_{i=1}^{d}\sigma_{i}^{2}b_{1}}t\right)  dt+1.
\end{align*}
Hence, using this estimate, together with (\ref{B_Aux_LL}) and (\ref{BBB}) we
conclude that
\[
E^{1/2}\left(  \left\Vert \mathbf{Y}\left(  \infty\right)  \right\Vert
_{1}^{2}\right)  E^{1/2}\left(  \left(  1-\beta_{0}\right)  ^{2\mathcal{N}%
\left(  t;\mathbf{Y}\left(  \infty\right)  \right)  }\right)  \leq3\cdot
d^2\cdot\frac{\kappa_{0}}{\delta_{0}\beta_{0}}b_{0}b_1%
\exp\left(  -\frac{\zeta_{1}}{2(d^{4}\log\left(  d\right)  )}t\right)  .
\]
On the other hand, directly from Proposition \ref{prop:A2}, we obtain (with
the same selection of $\beta$, in particular $\beta\in\left(  0,1/3\right)  $)
that
\[
\left\Vert \mathbf{y}\right\Vert _{1}E\left(  \left(  1-\beta_{0}\right)
^{\mathcal{N}\left(  t;\mathbf{y}\right)  }\right)  \leq3\cdot\left\Vert
\mathbf{y}\right\Vert _{1}\exp\left(  \zeta_{0}\left\Vert \mathbf{y}%
\right\Vert _{\infty}/(d^{3}\log\left(  d\right)  )+\beta/d^{2}\right)
\cdot\exp\left(  -\zeta_{1}t/(d^{4}\log\left(  d\right)  )\right)  .
\]
Putting these estimates together in (\ref{B}), we obtain that
\begin{align*}
&  E\left\vert f\left(  \mathbf{Y}\left(  t;\mathbf{y}\right)  \right)
-f\left(  \mathbf{Y}\left(  t;\mathbf{Y}\left(  \infty\right)  \right)
\right)  \right\vert \\
&  \leq3\cdot d^2\cdot\exp\left(  -\frac{\zeta_{1}}{d^{4}\log\left(  d\right)
}t\right)  \left(  \left\Vert \mathbf{y}\right\Vert _{1}\cdot\kappa_{0}%
\cdot\exp\left(  \zeta_{0}\frac{\left\Vert \mathbf{y}\right\Vert _{\infty}%
}{d^{3}\log\left(  d\right)  }\right)  +\frac{\kappa_{0}}{\delta_{0}\beta_{0}}b_{0}b_1\right)  .
\end{align*}

\section{Conclusions and Final Remarks}

\label{Sec: conclude}

We have shown that the relaxation time of the Harrison-Reiman RBM, under the
uniformity conditions A1)-A3), is polynomial in the underlying dimension $d$.
It is of interest to ponder how one may improve upon or extend the type of analysis that
we have presented.

We have used the
Wasserstein distance of order one only for convenience and because it
already covers functions (such as workload, maximum workload, average
workload), which are natural in practice. In terms of applications to Monte
Carlo, we note that most of the literature on rigorous analysis of rates of
convergence to stationarity actually focuses on total variation convergence
or related notions, which in some sense (e.g. in the sense of dealing with
moments of higher order) are even more restrictive than the notion that we
consider. So, we think that our notion strikes the right balance
between convenience in terms of tractability and flexibility in terms of
applicability.  We also note that we could actually handle
locally Lipschitz functions, i.e. functions satisfying that there exists a function $\kappa(\cdot)>0$ such that 
\begin{equation*}
\left\vert f\left( x\right) -f\left( x^{\prime }\right) \right\vert \leq
\kappa \left( x\right) \left\Vert x-x^{\prime }\right\Vert_{1}, \text{for all }x, x'\in\mathbb{R}^d.
\end{equation*}%
Assume that $E\left[ \kappa \left( \mathbf{Y}(\infty )\right) ^{p}\right] <\infty $
for any $p>0$. In this case, we have that 
\begin{equation*}
E\left\vert f\left( \mathbf{Y}(\infty)\right) -f\left( \mathbf{Y}(t)\right) \right\vert \leq
 E\left[\kappa \left( \mathbf{Y}(\infty)\right) ^{p}\right]^{1/p}\cdot E\left[\left\Vert
\mathbf{Y}(\infty)-\mathbf{Y}(t)\right\Vert _{1}^{q}\right]^{1/q}
\end{equation*}%
with $1/p +	1/q =1$. Note that $\|\mathbf{Y}(\infty)\|_1$ can be bounded by the sum of exponential random variables ($R^{-1}\mathbf{Y}^+(\infty)$) and hence $E[\|\mathbf{Y}(\infty)\|^{p}_1]$ is finite and polynomial in $d$ for any $p$. Therefore, our approach can be easily adapted to estimate an explicit bound for
$
E\left[\left\Vert \mathbf{Y}(\infty)-\mathbf{Y}(t)\right\Vert _1^{q}\right]
$
and for $E\left[\kappa \left( \mathbf{Y}(\infty)\right) ^{p}\right]$ with $f$ that is a quadratic or polynomial function such as $f(x)=x'Ax$ where $A$ is a matrix. 

The main bottleneck in our analysis arises from Step 2 outlined in Section
\ref{Sect_Outline_Strategy}. Namely, our analysis of $E\left[  \left(
1-\beta\right)  ^{\mathcal{N}\left(  t,\mathbf{y}\right)  }\right]  $, where
$\mathcal{N}\left(  t,\mathbf{y}\right)  $ counts the number of
\textquotedblleft tours\textquotedblright\ completed by $\mathbf{Y}\left(
\cdot;\mathbf{y}\right)  $ in the interval $[0,t]$. The first tour starts at
time zero, and we let one unit of time elapse. Then we terminate the tour at
the first time, $\eta^{1}\left(  \mathbf{y}\right)  $, by which all the
coordinates have hit zero at least once. The second tour starts right at time
$\eta^{1}\left(  \mathbf{y}\right)  $ and it proceeds just as we indicated for
the first tour and so on. The time $\eta^{1}\left(  \mathbf{y}\right)  $ is
therefore basically the maximum of $d$ stopping times $\eta_{k}^{1}%
(\mathbf{y})$ ($k=1,2,...,d$), corresponding to the first time at which
certain coordinate hits zero in the tour. We upper bound $\eta^{1}\left(
\mathbf{y}\right)  $ by the sum of $d$ i.i.d. random variables. This probably
has the effect of slowing the count of $\mathcal{N}\left(  t,\mathbf{y}%
\right)  $ by a factor of $d$ and therefore, at this point, we need to run the
process $O\left(  d\right)  $ units of time to compensate for this factor.

In addition, in our analysis of $\mathcal{N}\left(  t,\mathbf{y}\right)  $, we
estimate the time it takes for the underlying RBM to visit a compact set
around the origin (say all of the coordinates being less than unity). We
introduce an upper bound process with orthogonal reflection and construct a
Lyapunov function to estimate the moment generating function of the time
$\tau^{+}\left(  \mathbf{y}\right)  $ -- the time it takes for the upper bound
process $\mathbf{Y}^{+}$ to visit the compact set.

Our choice of the Lyapunov function is appropriate for bounding the moment
generating function of $\tau^{+}\left(  \mathbf{y}\right)  $ in a neighborhood
of the origin. Intuitively, this amounts to roughly capturing the behavior of
$E\left(  \tau^{+}\left(  \mathbf{y}\right)  \right)  $, which behaves (up to
constants) as $\left\Vert \mathbf{y}\right\Vert _{\infty}$. Since the bound is
obtained with the idea of allowing the dimension grow arbitrarily large, the
Lyapunov function introduces a scaling of the form $\theta\left\Vert
\mathbf{y}\right\Vert _{\infty}$ for $\theta=O\left(  1/d\right)  $, to
normalize the contribution of $\left\Vert \mathbf{y}\right\Vert _{\infty}$. In
addition, we introduce a quadratic behavior in the construction of the
Lyapunov function to deal with the reflective boundaries, and a mollification
parameter $\varepsilon$ to smoothly approximate the (non-smooth) function
$\left\Vert \mathbf{y}\right\Vert _{\infty}$. Our choice of $\theta$ forces us
to choose $\varepsilon=O\left(  1/(d^{2}\log\left(  d\right)  )\right)  $. The
combination of all of these factors (and other algebraic manipulations) leads
to the relaxation of time roughly of order $O\left(  d^{4}\right)  $
(neglecting logarithmic factors).

In the case of $Q=0$, under our current assumptions, one can deal directly
with the maximum of the times for the coordinates to hit zero, and there is no
need for introducing an upper bound process. In such a case, the relaxation
time is easily seen to be $O\left(  \log\left(  d\right)  \right)  $.

We believe that under additional structural assumptions, for example in the
setting of feedforward networks ($Q$ is a triangular matrix with at least one row of zeroes), it is likely that the techniques introduced
in this paper can be used to show that the relaxation time can be
significantly improved. We plan to pursue these investigations in future
work

Finally, we believe that the basic structure of our proof technique may be
applicable to other processes beyond RBM. The key parts of the analysis
involve (1) Lemma 2 that corresponds to Step 1 in the roadmap of the proof,
and (2) the introduction of a system which serves as an upper bound (in our case
a system with orthogonal reflection). Part (2) is crucial to the construction
of estimates for the return time to a suitably defined compact set (corresponding to
Step 2 in the roadmap). Under analogous assumptions to those discussed in
this paper, these parts can be obtained in the setting of time-varying RBM
with periodic input, if one is interested in estimating steady-state
expectations for the discrete process sampled along integer multiples of the
period. Another setting in which these parts can also be obtained is that of
generalized Jackson networks. The upper bound process in that case, however,
should be defined in terms of a so-called autonomous network (see, for
example, \cite{BlanchetChen_2018} and the references therein).\newline



\begin{thebibliography}{99}

\bibitem{MLMC}
J.~Blanchet, X.~Chen, P.~Glynn, and N.~Si.
\newblock Efficient steady-state simulation of reflected Brownian motion.
\newblock \emph{Working paper}, 2018.

\bibitem{BlanchetChen_2018}
J.~Blanchet and X.~Chen.
\newblock Perfect sampling of generalized Jackson networks.
\newblock \emph{Mathematics of Operations Research}, forthcoming.

\bibitem{Creemer_2010}
S.~Creemers and M.~Lambrecht.
\newblock Modeling a hospital queueing network.
\newblock In \emph{Queueing Networks: International Series in Operations Research \& Management Science}.
  Springer-Verlag, 2010.

\bibitem{Armony_2018}
M.~Armony, S.~Israelit, A.~Mandelbaum, Y.~N. Marmor, Y.~Tseytlin, and G.~B. Yom-Tov.
\newblock On patient flow in hospitals: a data-based queueing-science perspective.
\newblock \emph{Stochastic Systems}, 5(1):146--194, 2015.

\bibitem{Feller_1968}
W.~Feller.
\newblock \emph{An Introduction to Probability Theory and its Applications}.
\newblock John Wiley \& Sons, 1968.

\bibitem{BudhirajaLee_2007}
A.~Budhiraja and C.~Lee.
\newblock Long time asymptotics for constrained diffusions in polyhedral domains.
\newblock \emph{Stochastic Processes and their Applications}, 117(8):1014--1036, 2007.

\bibitem{Sarantsev_2017}
A.~Sarantsev.
\newblock Reflected Brownian motion in a convex polyhedral cone: tail estimates for the stationary distribution.
\newblock \emph{Journal of Theoretical Probability}, 30(3):1200--1223, 2017.

\bibitem{Harrison_1992}
J.~M. Harrison and R.~J. Williams.
\newblock Brownian models of feedforward queueing networks: quasireversibilty and product form solutions.
\newblock \emph{The Annals of Applied Probability}, 2:263--293, 1992.

\bibitem{Peterson_1990}
W.~P. Peterson.
\newblock A heavy traffic limit theorem for networks of queues with multiple customer types.
\newblock \emph{Mathematics of Operations Research}, 9:90--118, 1991.

\bibitem{Blanchet_Chen_2015}
J.~Blanchet and X.~Chen.
\newblock Steady-state simulation of reflected Brownian motion and related stochastic networks.
\newblock \emph{The Annals of Applied Probability}, 25:3209--3250, 2015.

\bibitem{Williams_1995}
R.~J. Williams.
\newblock Semimartingale reflecting Brownian motions in the orthant.
\newblock In \emph{Stochastic Networks, the IMA Volumes in Mathematics and its Applications}, volume~71,
  pages 125--137, 1995.

\bibitem{HarrisonWilliams_1987b}
J.~M. Harrison and R.~J. Williams.
\newblock Brownian models of open queueing networks with homogeneous customer populations.
\newblock \emph{Stochastics}, 2:77--115, 1987.

\bibitem{Kella_1996}
O.~Kella.
\newblock Stability and nonproduct form of stochastic fluid networks with L\'{e}vy inputs.
\newblock \emph{The Annals of Applied Probability}, 6:186--199, 1996.

\bibitem{Budhirajaetal2014}
A.~Budhiraja, J.~Chen, and S.~Rubenthaler.
\newblock A numerical scheme for invariant distributions of constrained diffusions.
\newblock \emph{Mathematics of Operations Research}, 39:262--289, 2014.

\bibitem{Bovier2010}
A.~Bovier.
\newblock Extremes, sums, L\'{e}vy processes, and ageing.
\newblock Lectures given in 2010 at the Technion, Haifa.

\bibitem{Whitt2002}
W.~Whitt.
\newblock \emph{Stochastic-Process Limits}.
\newblock Springer-Verlag, 2002.

\bibitem{KellaWhitt_1996}
O.~Kella and W.~Whitt.
\newblock Stability and structural properties of stochastic storage networks.
\newblock \emph{Journal of Applied Probability}, 33:1169--1180, 1996.

\bibitem{AGP1995}
S.~Asmussen, P.~Glynn, and J.~Pitman.
\newblock Discretization error in simulation of one-dimensional reflecting Brownian motion.
\newblock \emph{The Annals of Applied Probability}, 5:875--896, 1995.

\bibitem{MandelbaumRamanan_2010}
A.~Mandelbaum and K.~Ramanan.
\newblock Directional derivatives of oblique reflection maps.
\newblock \emph{Mathematics of Operations Research}, 35:527--558, 2010.

\bibitem{BlanchetMurthy2015}
J.~Blanchet and K.~Murthy.
\newblock Exact simulation of multidimensional reflected Brownian motion.
\newblock Submitted.

\bibitem{KellaRamasubramanian_2012}
O.~Kella and S.~Ramasubramanian.
\newblock Asymptotic irrelevance of initial conditions for Skorokhod refection mapping on the nonnegative orthant.
\newblock \emph{Mathematics of Operations Research}, 37:301--312, 2012.

\bibitem{BlanchetChen_2015}
J.~Blanchet and X.~Chen.
\newblock Steady-state simulation of reflected Brownian motion and related stochastic networks.
\newblock \emph{The Annals of Applied Probability}, 25:3209--3250, 2015.

\bibitem{DaiHarrison_1992}
J.-G. Dai and J.~M. Harrison.
\newblock Reflected Brownian motion in an orthant: numerical methods for steady-state analysis.
\newblock \emph{The Annals of Applied Probability}, 2:65--86, 1992.

\bibitem{HarrisonWilliams_1987}
J.~M. Harrison and R.~J. Williams.
\newblock Multidimensional reflected Brownian motions having exponential stationary distributions.
\newblock \emph{The Annals of Probability}, 15:115--137, 1987.

\bibitem{BudhirajaLee_2009}
A.~Budhiraja and C.~Lee.
\newblock Stationary distribution convergence for generalized Jackson networks in heavy traffic.
\newblock \emph{Mathematics of Operations Research}, 34:45--56, 2009.

\bibitem{GamarnikZeevi_2006}
D.~Gamarnik and A.~Zeevi.
\newblock Validity of heavy traffic steady-state approximations in generalized Jackson networks.
\newblock \emph{The Annals of Applied Probability}, 16:56--90, 2006.

\bibitem{HarrisonReiman_1981}
J.~M. Harrison and M.~I. Reiman.
\newblock Reflected Brownian motion on an orthant.
\newblock \emph{The Annals of Applied Probability}, 9:302--308, 1981.

\bibitem{ChenYao2001}
H.~Chen and D.~D. Yao.
\newblock \emph{Fundamentals of Queueing Networks: Performance, Asymptotics and Optimization}.
\newblock Springer Verlag, 2001.

\end{thebibliography}
\end{document}